\documentclass[11pt, a4paper, reqno]{amsart}
\usepackage{mathrsfs}
\usepackage{amsfonts}
\usepackage{amsmath}
\usepackage{graphicx}
\usepackage{amssymb}
\usepackage{amsthm}
\usepackage{color}
\usepackage{hyperref}
\usepackage{vmargin}

\long\def\symbolfootnote[#1]#2{\begingroup%
\def\thefootnote{\fnsymbol{footnote}}\footnote[#1]{#2}\endgroup}

\setmarginsrb{20mm}{20mm}{20mm}{20mm}{10mm}{10mm}{10mm}{10mm}

\newtheoremstyle{remark}
  {}{}{}{}{\bfseries}{.}{.5em}{{\thmname{#1 }}{\thmnumber{#2}}{\thmnote{ (#3)}}}

\DeclareMathOperator\dist{dist}

\RequirePackage{amsthm}
\newtheorem*{theorem*}{Theorem}
\newtheorem{theorem}{Theorem}[section]
\newtheorem{lem}[theorem]{Lemma}
\newtheorem{cor}[theorem]{Corollary}
\newtheorem{prop}[theorem]{Proposition}

\theoremstyle{definition}

\theoremstyle{remark}
\newtheorem{rem}[theorem]{Remark}
\newtheorem{ex}{Example}

\def\div{{\rm div}}
\def\vint{\mathop{\mathchoice%
          {\setbox0\hbox{$\displaystyle\intop$}\kern 0.22\wd0%
           \vcenter{\hrule width 0.6\wd0}\kern -0.82\wd0}%
          {\setbox0\hbox{$\textstyle\intop$}\kern 0.2\wd0%
           \vcenter{\hrule width 0.6\wd0}\kern -0.8\wd0}%
          {\setbox0\hbox{$\scriptstyle\intop$}\kern 0.2\wd0%
           \vcenter{\hrule width 0.6\wd0}\kern -0.8\wd0}%
          {\setbox0\hbox{$\scriptscriptstyle\intop$}\kern 0.2\wd0%
           \vcenter{\hrule width 0.6\wd0}\kern -0.8\wd0}}%
          \mathopen{}\int}

\newcommand{\Om}{\Omega}
\newcommand{\C}{\mathbb{C}}
\newcommand{\R}{\mathbb{R}}

\newcommand{\ep}{\epsilon}
\newcommand{\p}{{$p\mspace{1mu}$}}

\newcommand{\Ric}{{\rm Ric}}

\newcommand{\vp}{\varphi}
\newcommand{\si}{\sigma}
\newcommand{\ud}{\mathrm {d}}

\definecolor{blau}{rgb}{0.1,0.0,0.9}
\definecolor{violet}{rgb}{0.54, 0.17, 0.89}
\newcommand{\blue}{\color{blau}}

\newcommand{\kom}[1]{}
\renewcommand{\kom}[1]{{\bf \blue /#1/}}

\newcounter{komcounter}
\numberwithin{komcounter}{section}

\title[]{Isoperimetric inequalities and geometry of level curves of harmonic functions on smooth and singular surfaces
}

\author[]{Tomasz Adamowicz{\small$^1$}}
\address{T.A.: Institute of Mathematics, Polish Academy of Sciences,
\'Sniadeckich 8, Warsaw, 00-656, Poland\/}
\email{tadamowi@impan.pl}
\author[]{Giona Veronelli}
\address{G.V.: Dipartimento di Matematica e Applicazioni, Università di Milano Bicocca, via R. Cozzi 55, 20126 Milano, Italy}
\email{giona.veronelli@unimib.it}

\begin{document}


\footnotetext[1]{T. Adamowicz was supported by a grant of National Science Center, Poland (NCN),
 UMO-2017/25/B/ST1/01955.}
\begin{abstract}
 We investigate the logarithmic convexity of the length of the level curves for harmonic functions on surfaces and related isoperimetric type inequalities. The results deal with smooth surfaces, as well as with singular Alexandrov surfaces (also called surfaces with bounded integral curvature), a class which includes for instance surfaces with conical singularities and surfaces of CAT(0) type.  Moreover, we study the geodesic curvature of the level curves and of the steepest descent for harmonic functions on surfaces with non-necessarily constant Gaussian curvature $K$. Such geodesic curvature functions turn out to satisfy certain Laplace-type equations and inequalities, from which we infer various maximum and minimum principles.
  
 The results are complemented by a number of growth estimates for the derivatives $L'$ and $L''$ of the length of the level curve function $L$, as well as by examples illustrating the presentation.
 
  Our work generalizes some results due to Alessandrini, Longinetti, Talenti, Ma--Zhang and Wang--Wang.
\end{abstract}

\maketitle

\section{Introduction}

 The geometry of level sets has been a vital and fruitful topic of investigations involving variety of function and space settings. In particular, in the Euclidean spaces the convexity properties and curvature estimates of level sets of harmonic functions and their generalizations (such as $p$-harmonic functions and mappings and second order elliptic PDEs) have been studied by many researchers, for instance by Alessandrini~\cite{al, al2}, Caffarelli--Spruck~\cite{cs}, Chang--Ma--Yang~\cite{cmy}, Gabriel~\cite{gab}, Jost--Ma--Ou~\cite{jmo}, Kawohl~\cite{kaw}, Laurence~\cite{la}, Lewis~\cite{lew3}, Longinetti~\cite{long}, Ma--Ou--Zhang~\cite{moz}, Talenti~\cite{tal}; see also~\cite{ad}. Similar studies have been conducted for harmonic functions in space forms by Ma--Zhang~\cite{mz} and on Riemannian surfaces with constant Gaussian curvature by Wang--Wang~\cite{ww}.
 
 The main goal of this note is to extend number of the aforementioned results to the setting of Riemannian $2$-manifolds with not necessarily constant curvature, including the surfaces of bounded integral curvature (Alexandrov surfaces). The latter one is a class of singular spaces which includes the polyhedral surfaces, the surfaces with conical singularities as well as all the topological metric surfaces with (one sided) bounds on the curvature in the sense of Alexandrov, such as $CBB(k)$ and $CAT(k)$, or $RCD(k,2)$ surfaces. This topic has been attracting an increasing interest in the recent decades, for instance due to its connections with the study of Gromov--Hausdorff limits of manifolds with bounded curvature. 
 

Moreover, we present a number of growth estimates and maximum (minimum) principles for the geodesic curvature of level sets and the curvature of the steepest descent for harmonic functions.
 
Let us now present the main results and the organization of the paper. In {\bf Section 2} we consider the harmonic Dirichlet problem on an annular domain of a Riemannian surface $M^2$:
\begin{equation}\label{DP0}\tag{DP}
	\begin{cases}
		\Delta_{M} u=0 & \hbox{in } \Om,\\
		u|_{\Gamma_1}=t_1, & u|_{\Gamma_2}=t_2,
	\end{cases}
\end{equation}
assuming constant data $t_1$ and $t_2$ on the two boundary components $\Gamma_1$ and $\Gamma_2$, respectively. We show that in the smooth setting the \emph{isoperimetric inequality}
\begin{equation}\label{convex intro}
(\ln L(t))''\geq 0,\quad t_1\leq t \leq t_2,
\end{equation}
for the length of the level curves $L(t)$ characterizes the surfaces with nonpositive Gaussian curvature. Moreover, the equality in \eqref{convex intro} forces the domain to be a flat round annulus; see Theorems~\ref{thm-main1} and~\ref{thm-main2}.  The fact that the key differential inequalities studied here become equalities depending on the geometry of domains, motivates the name isoperimetric inequality in the title of our work. Indeed, we follow the terminology in~\cite{al2, long, la}, where such results for~\eqref{DP} are shown in the Euclidean spaces; see also~\cite{lew3} for importance of the capacitary estimates related to the problem~\eqref{DP}. The generalization to non-flat surfaces requires obtaining integral formulas for the derivatives $L'$ and $L''$: see Lemma~\ref{lem:  L''} whose proof relies on the coarea formula, the Bochner and the Kato identities. Our discussion is complemented by  
refined lower estimates on $(\ln L)''$ for surfaces with bounds for the Gaussian curvature (e.g. surfaces with pinched curvatures), see Propositions 2.9 and~\ref{thm-main2-2}. We illustrate the presentation by analyzing some level curves in a hyperbolic surface; see Example 1. An alternative more analytic proof of Theorem \ref{thm-main1}, based on a conformal approach, is presented in Appendix A. We also remark here that inequalities similar to~\eqref{convex intro} appear in the context of the Hadamard's three-circles theorems (see Remark \ref{rmk:AG}) and studies of the K\"ahler manifolds, see e.g.~\cite[Theorems 4.2 and 2.2]{mac}.

 Furthermore, for non-smooth \emph{surfaces with non-negative bounded integral curvature} we show that the function $t\to \ln L(t)$ is convex (Theorem~\ref{thm-reshetnyak} in {\bf Section 3}). Since $L(t)$ is no more a smooth function in this setting, such a result requires a different approach based on the measure-theoretic results due to Reshetnyak~\cite{re} and Troyanov~\cite{tr-AnnIHP}. As a byproduct of our studies we obtain that in the special case of zero curvature, Theorem~\ref{thm-reshetnyak} allows to study also annuli with non-smooth boundary components (for example von Koch snowflakes) and thus the result is new also in the Euclidean spaces; see Remark~\ref{r: Jordan bdy}. Among the surfaces with bounded integral curvature to which our result applies, let us point out in particular the important class of the \emph{surfaces with conical singularities} whose angles at the vertices are greater than $2\pi$ (i.e. of non-positive curvature); see Section 3. 
 
The aim of {\bf Section 4} is to extend some results in~\cite{al, tal}, obtained for domains in $\R^n$, to the setting of domains in $2$-manifolds and for harmonic functions with no critical points in the underlying domain (see the discussion following Theorem~\ref{thm-pdes} for the feasibility of this assumption). Moreover, we generalize certain results proved in~\cite{mz, ww} for surfaces with constant Gaussian curvature, by allowing the curvature to vary.
Specifically, let $u$ be a harmonic function in a domain $\Om$ in a surface $M^2$ with Gaussian curvature $K=K(x)$ for $x\in M^2$, such that $u$ has no critical points in $\Om$. Then, Theorem~\ref{thm-pdes} says that the geodesic curvature $k$ of the level curves of $u$ satisfies:
 \begin{align*}
 &\Delta \left(\frac{k}{|\nabla u|}\right)+2K\frac{k}{|\nabla u|}=\frac{\langle \nabla K,\nabla u\rangle}{|\nabla u|^2}, \\
 &-\Delta \ln |k| \geq K-\frac{1}{|k|}\langle \nabla K,\frac{\nabla u}{|\nabla u|}\rangle,\quad k\not=0.
\end{align*}
Similar differential (in)equalities hold true replacing $k$ with the curvature $h$ of the steepest descent of $u$. Namely:
 \begin{align*}
 	& \Delta \left(\frac{h}{|\nabla u|}\right)+2K\frac{h}{|\nabla u|}=-\frac{\langle \nabla K,\star\nabla u\rangle}{|\nabla u|^2}, 
  \\
 	&-\Delta \ln |h| \geq K+\frac{1}{|h|}\langle \nabla K,\frac{\star\nabla u}{|\nabla u|}\rangle,\quad h\not=0,
 \end{align*}
where $\star$ stands for the Hodge star operator (so that with respect to a local orthonormal frame one has $\star \nabla u=(u_2,-u_1)$). Our Theorem~\ref{thm-pdes} generalizes Theorem 1.3 in~\cite{ww} and Theorem 3 in~\cite{tal} to the non-constant curvature setting. Furthermore, we study more types of curvatures of level sets than in~\cite{ww}, see the discussion before the statement of Theorem~\ref{thm-pdes} for a detailed presentation of the novelties obtained in Section 4. The corollaries of Theorem~\ref{thm-pdes} encompass weak and strong maximum and minimum principles for the curvatures of the level curves (Corollaries~\ref{c: k}\,--\,\ref{c:cor3.7}). The proof of Theorem~\ref{thm-pdes} is presented in Appendix B.
\vspace{0.2cm}

{\bf Acknowledgements.} Part of the work was conducted during the Simons semester in \emph{Geometry and analysis in function and mapping theory on Euclidean and metric measure spaces} at IMPAN in fall 2019 partially supported by the grant 346300 for IMPAN from the Simons Foundation and the matching 2015-2019 Polish MNiSW fund.
The second author is member of INdAM-GNAMPA. The authors would like to thank Luciano Mari and Tommaso Pacini for their comments about the manuscript and for pointing some literature.
 \section{Isoperimetric inequality}

 The goal of this section is to show a counterpart of Alessandrini's isoperimetric inequality result~\cite[Theorem 1.1]{al2} for harmonic functions on two-dimensional Riemannian manifolds. However, we will formulate the main problem for all dimensions $n\geq 2$, as some of our results below can be applied in the general case of smooth Riemannian manifolds.
 
 Let $(M^n,g)$ be an $n$-dimensional Riemannian manifold with the Ricci curvature bounded from below: $\Ric\geq c$, for some fixed $c\in \R$. Moreover, let $\Om_1, \Om_2$ with $\Om_1\Subset \Om_2\subset M^n$ be two simply connected domains with boundaries denoted by $\Gamma_1$ and $\Gamma_2$, respectively. In what follows we will either assume that $\Gamma_1$ and $\Gamma_2$ are Jordan curves (mostly for $2$-dimensional results) or require
them to be $C^{1,\alpha}$ (mostly for $n>2$), see Remark~\ref{rem-c1a-reg} for a further discussion. We will specify the boundary regularity assumptions when stating the results. We define a topological annulus in $M^n$ (i.e. a ring domain, sometimes also called in the literature a $2$-connected domain when $n=2$) by
 \begin{equation}\label{def-top-ring}
 \Om:=\Om_2\setminus \overline{\Om_1}.
 \end{equation}
 Furthermore, let $t_1,t_2\in \R$ be such that $t_1<t_2$ and let us consider a continuous up to the boundary solution of the following Dirichlet problem in $\Om$ for the Laplace--Beltrami harmonic operator $\Delta_M$ on $M^n$:
 \begin{equation}\label{DP}\tag{DP}
 \begin{cases}
 \Delta_{M} u=0 & \hbox{in } \Om,\\
 u|_{\Gamma_1}=t_1, & u|_{\Gamma_2}=t_2.
 \end{cases}
 \end{equation}
 Since now on we adopt a convention that for a fixed manifold $M$ and a fixed metric $g$ we will omit the subscript $M$ in $\Delta_M$ and write $\Delta$ for simplicity. Furthermore, when discussing $2$-dimensional manifolds we use the Gaussian curvature $K$ instead of the Ricci curvature, since $\Ric(\nabla u, \nabla u)\equiv K|\nabla u|^2$ for $n=2$.
 
 The following auxiliary result is a counterpart of the well-known subharmonicity property for harmonic functions in the Euclidean setting.
\begin{lem}\label{lem:log}
 Let $\Om$ be a domain in a $2$-dimensional Riemannian manifold and $u$ be a harmonic function in $\Om$. Then,  
 \begin{equation}\label{log-sub-ineq}
 \Delta (\log|\nabla u|)=K
 \end{equation}
at points, where $|\nabla u|\neq 0$. In particular, if $K|_{\Om}\geq 0$, then $
  \Delta(\log|\nabla u|)\geq 0$, while if $K|_{\Om}\leq 0$, then $
  \Delta(\log|\nabla u|)\leq 0$.
\end{lem}
Before proving the lemma we recall the refined Kato (in)equality, a standard tool in geometric analysis (see for instance \cite[p. 520]{lw} or \cite[Proposition 1.3]{PRS}  and references therein).  In dimension $2$, the refined Kato's inequality turns out to be an equality; see  \eqref{refined-Kato2}. Although this observation might be known to experts, we were not able to find a reference in the literature. Accordingly, we decided to provide a complete proof. 
\begin{lem}\label{lem: kato}
	Let $\Om$ be a domain in an $n$-dimensional Riemannian manifold and $u$ be harmonic in $\Om$. Then, at points where $|\nabla u|\neq 0$, it holds that
	\begin{equation}\label{refined-Kato}
	|\nabla^2u|^2\geq \frac{n}{n-1}\left|\nabla|\nabla u|\right|^2,\quad n\geq 2.
	\end{equation}
	Moreover, if $n=2$ then one has indeed  
	\begin{equation}\label{refined-Kato2}
|\nabla^2u|^2=2\left|\nabla|\nabla u|\right|^2.
\end{equation}
\end{lem}
 \begin{proof}
 	We prove \eqref{refined-Kato} and \eqref{refined-Kato2} at an arbitrary point $p\in \Omega$. Let us consider a normal coordinate system $(x_1,\dots,x_n)$ defined in a neighbourhood of $p$. Since Christoffel's symbols vanish at $p$, we have that $(\nabla^2 u)_{ij}(p)= \partial_i\partial_j u$ is a symmetric matrix. Accordingly, there exists an orthonormal basis $\{e_j\}_{j=1}^n$ of $T_pM$ which diagonalizes $(\nabla^2 u)_{ij}(p)$. Note that $\{e_j\}_{j=1}^n$ can be obtained by $\{\partial_j\}_{j=1}^n$ via an orthonormal matrix $A$. In particular, composing with $A$ we get a new  normal coordinate system $(y_1,\dots,y_n)$ defined in a neighbourhood of $p$ with respect to which:
 	\begin{itemize}
 	\item	$(0,\dots,0)$ represents $p$,
 	\item $g_{ij}(p)=\delta_{ij}$ and $g^{ij}(p)=\delta^{ij}$. Hence, in particular $|\nabla u|^2(p)=\sum_{j=1}^n(\partial_ju)^2=:\sum_{j=1}^nu_j^2$.
 	\item $\partial_ig_{jk}(p)=0$ and $\Gamma_{jk}^i(p)=0$. Hence, in particular $(\nabla^2 u)_{ij}(p)= \partial_i\partial_j u=:u_{ji}$.
 	\item $(\nabla^2 u)_{ij} (p) =u_{ij}(p)= \lambda_i\delta_{ij}$ for real numbers $\lambda_1,\dots,\lambda_n$ satisfying $\sum_{j=1}^n\lambda_j =0$ (because $\Delta u=0$). 	
 	\end{itemize}
 In any point of this coordinate system we have $|\nabla u|^2=\sum_{i,j=1}^nu_iu_j g^{ij}$, so that at $p$
 \begin{align*}
 (\nabla |\nabla u|^2)_k(p)&=\sum_{i,j=1}^n\Big(u_{ik}(p)u_{j}(p) g^{ij}(p)+u_{i}(p)u_{jk}(p) g^{ij}(p)+u_{i}(p)u_{j}(p) \partial_k g^{ij}(p)\Big)\\
 &=\sum_{i,j=1}^n\Big(\lambda_k\delta_{ik}u_{j}(p) \delta^{ij}+u_{i}(p)\lambda_k\delta_{jk}(p) \delta^{ij}\Big)\\
 &=2\lambda_k u_k(p),
 \end{align*}
 and 
 \begin{equation}\label{hess1}
 |\nabla u|^2(p)|\nabla |\nabla u||^2(p)=\frac{1}{4}|\nabla |\nabla u|^2|^2(p)= \sum_{k=1}^n \lambda_k^2 u^2_k(p).
 \end{equation}
 On the other hand
 \[
 |\nabla^2u|^2(p)=\sum_{i,j,k,l=1}^n (\nabla^2u)_{ij}(p)(\nabla^2u)_{kl}(p) g^{ik}(p)g^{jl}(p)=\sum_{k=1}^n \lambda_k^2,
 \]
 so that 
 \begin{equation}\label{hess2}
|\nabla u|^2(p)|\nabla^2 u|^2(p) = \sum_{k=1}^n\lambda_k^2\sum_{k=1}^n u_k^2(p). 
 \end{equation}
 Now, if $n=2$, we have $\lambda_2^2=(-\lambda_1)^2$ and so 
 \[
|\nabla u|^2(p)|\nabla^2 u|^2(p) = 2\lambda_1^2(u_1^2(p)+u_2^2(p))=2(\lambda_1^2 u^2_1(p)+\lambda_2^2 u^2_2(p))= 2|\nabla u|^2(p)|\nabla |\nabla u||^2(p).  
 \]
 From this, identity~\eqref{refined-Kato2} follows immediately. For general $n>2$, since $\sum_{k=1}^n\lambda_k=0$, we get for every $j$,
 \[
 \lambda_j^2=\frac{n-1}{n}\lambda_j^2 + \frac{1}{n}\Big(\sum_{k\neq j}\lambda_k\Big)^2 \leq \frac{n-1}{n}\lambda_j^2 + \frac{1}{n}(n-1)\sum_{k\neq j}\lambda^2_k =  \frac{n-1}{n}\sum_{k=1}^n\lambda^2_k.
 \]
 By combining these inequalities with \eqref{hess1} and \eqref{hess2} we obtain estimate~\eqref{refined-Kato} and conclude the proof.
 \end{proof}
\begin{proof}[Proof of Lemma \ref{lem:log}]
 Since all assertions of the lemma are pointwise, we may apply the Korn--Lich\-ten\-stein theorem and locally change the coordinates to the conformal (isothermal) ones. The existence of the conformal coordinates on the Riemannian surfaces can be proven, e.g. by employing the Beltrami equation. Then, locally the given metric $g$ reads $g=\lambda(z)^2 dz^2$ for a smooth $\lambda>0$ and planar Euclidean coordinates $z=(x,y)$. Recall that the harmonicity of $u$ in dimension $2$ and the fact that $\nabla u$ is non-vanishing are conformal invariants. Thus, upon denoting the Laplacian and, respectively, the length of the gradient in the Euclidean coordinates by $\Delta_0$, respectively $|\nabla \cdot|_{0}$, the following computations hold true, completing the proof of the lemma:
\begin{align*}
\Delta (\log|\nabla u|)=\lambda^{-2} \Delta_0 \log (\lambda^{-1} |\nabla u|_{0})&=-\lambda^{-2} \Delta_0 \log (\lambda)+\lambda^{-2} \Delta_0 \log (|\nabla u|_{0}) \\
&=-\lambda^{-2} \lambda^{2} (-K)+0 =K.
\end{align*}
\end{proof}

\begin{rem} In dimensions $n>2$ the above approach is not efficient. However, one may instead prove that $\Delta |\nabla u|^\alpha\geq 0$ for $\alpha\geq \frac{n-2}{n-1}$ and $n\geq 2$ when the Ricci curvature is non-negative. Indeed, a direct computation using the curvature assumption, the Bochner formula \eqref{formula-Bochner} and the refined Kato inequality gives
	\begin{align*}
\Delta 	|\nabla u|^\alpha &= \alpha |\nabla u|^{\alpha -2}\left[(\alpha -1)|\nabla|\nabla u||^2+|\nabla u|\Delta|\nabla u|\right]\\
& =\alpha |\nabla u|^{\alpha -2}\left[(\alpha -2)|\nabla|\nabla u||^2+
|\nabla^2 u|^2 + \Ric(\nabla u, \nabla u)
\right] \\
&\geq \alpha |\nabla u|^{\alpha -2}\Ric(\nabla u, \nabla u) \geq  0.
	\end{align*}
The relation $\Delta |\nabla u|^\alpha\geq \alpha \kappa |\nabla u|^\alpha$ on manifolds with Ricci curvature lower bounded by $\kappa$ was used for instance in \cite[Theorem 2.1]{lw} to obtain some structure theorems and in \cite{dindos, MeVe} to study the Hardy spaces.
\end{rem}

%
%

\begin{rem}\label{rem: hd}
 Suppose that in the setting of $n$-dimensional manifolds, for $n\geq 3$, one can prove the following variant of the Kato estimate
 \begin{equation}\label{oq}
 	 |\nabla^2u|^2\leq 2\left|\nabla|\nabla u|\right|^2+k|\nabla u|^2.
 	  \end{equation}
   Then, one could assume that $\Ric\leq k$ to obtain a variant of Lemma~\ref{lem:log} with the upper bound $\Delta(\log|\nabla u|)\leq 2k$. For this reason, we wonder whether there exists some class of assumptions implying the validity of \eqref{oq}.
\end{rem}
%
%
Recall that the measure of a level set of a function $v:\Om\to\R$ for a domain $\Om\subset M^n$ is given by
\begin{equation}\label{iso-in-length}
 L(t)=\int_{\{x\in \Om\,:\,v(x)=t\}}1\,{\rm d}\mathcal{H}^{n-1},
\end{equation}
where ${\rm d}\mathcal{H}^{n-1}$ stands for the {$(n-1)$-Hausdorff} measure.

In the next lemma we prove formulas allowing us to compute the first and the second derivatives of $L$ with respect to the height of the level curve of harmonic functions in a topological annulus $\Om\subset M^n$. The lemma generalizes Lemma 2.1 in~\cite{al2}.
\begin{lem}\label{lem: L''}
Suppose that  $u:\Om\to\R$  is a harmonic function satisfying $|\nabla u|>0$ in a topological annulus $\Om\subset M^n$ and that $u$ attains constant boundary values, respectively, $u|_{\Gamma_1}=t_1$ and $u|_{\Gamma_2}=t_2$.
Then, the following holds for all $t_1<t<t_2$:
\begin{align}
 L'(t)&=\int_{\{x\in \Om\,:\,u(x)=t\}} \div\left(\frac{\nabla u}{|\nabla u|}\right)\frac{{\rm d}\mathcal{H}^{n-1}}{|\nabla u|}
 =\int_{\{x\in \Om\,:\,u(x)=t\}} \left \langle -\frac{\nabla u}{|\nabla u|}, \frac{ \nabla |\nabla u|}{|\nabla u|^2} \right \rangle {\rm d}\mathcal{H}^{n-1}. \label{L-der1} 
 \end{align}
Moreover, 
	\begin{align}\label{L'' n}
	L''(t) &\leq \int_{\{x\in \Om\,:\,u(x)=t\}} \left(\frac{2n-3}{n-1}\right)\frac{\big|\nabla |\nabla u|\big|^2}{|\nabla u|^4} - \frac{\Ric(\nabla u, \nabla u)}{|\nabla u|^4} {\rm d}\mathcal{H}^{n-1}.
\end{align}
 Additionally, if $\Omega$ is a domain in a $2$-dimensional manifold with Gauss curvature $K=K(x)$, inequality \eqref{L'' n} holds in the following stronger form
 \begin{align}
 L''(t) &= \int_{\{x\in \Om\,:\,u(x)=t\}} \frac{\big|\nabla |\nabla u|\big|^2}{|\nabla u|^4} - \frac{K}{|\nabla u|^2} {\rm d}\mathcal{H}^1. \label{L-der2}   
 \end{align}
\end{lem}


In the proof of the lemma we employ the following Bochner formula for harmonic functions:
\begin{equation}\label{formula-Bochner}
\Delta \frac{|\nabla u|^2}{2}=|\nabla^2u|^2+\Ric(\nabla u, \nabla u).
\end{equation}
Moreover, upon noticing that
\[
\Delta \frac{|\nabla u|^2}{2}=\div\left(\nabla \frac{|\nabla u|^2}{2}\right)=|\nabla u|\Delta|\nabla u|+\left|\nabla|\nabla u|\right|^2
\]
we obtain
\begin{equation}\label{cor-Bochner}
 \frac{\Delta|\nabla u|}{|\nabla u|}=\frac{|\nabla^2u|^2-\left|\nabla|\nabla u|\right|^2
}{|\nabla u|^2}+\frac{\Ric(\nabla u, \nabla u)}{|\nabla u|^2}.
\end{equation}

\begin{proof}
Since $|\nabla u|>0$ by assumptions, then $\nu=\frac{\nabla u}{|\nabla u|}$ is a unit vector normal to the level sets of $u$. Therefore, by the definition of the function $L$ in~\eqref{iso-in-length} and the Stokes theorem, we have that
 \begin{align*}
  L'(t)=\lim_{\ep \to 0}\frac{1}{\ep}\left(\int_{\{u=t+\ep\}} 1- \int_{\{u=t\}} 1\right) &=\lim_{\ep \to 0}\frac{1}{\ep}\left(\int_{\{u=t+\ep\}} \left \langle \nu, \frac{\nabla u}{|\nabla u|} \right \rangle - \int_{\{u=t\}} \left \langle \nu, \frac{\nabla u}{|\nabla u|} \right \rangle \right) \\
  &= \lim_{\ep \to 0}\frac{1}{\ep} \int_{\{t<u<t+\ep\}} \div\left(\frac{\nabla u}{|\nabla u|}\right).
 \end{align*} 
 At this stage we notice that the following integrals converge in $L^1$, as $\ep\to 0$:
 \begin{equation}\label{conv-level-sets}
  \frac{1}{\ep}\int_{\{t<u<t+\ep\}} 1 \to \int_{\{u=t\}}\frac{1}{|\nabla u|}.
 \end{equation}
 Indeed, since the limit on the left hand side above equals $\frac{d}{dt}\left(\int_{\{u<t\}} 1\right)$, the claim follows from the coarea formula; see e.g. Proposition 3 and its proof in~\cite[Chapter 3.4.4]{EGbook} or Exc. III.12(c) in \cite[Ch. 3]{chavel}. Similarly, the proof of the same proposition implies that
 \[
 \lim_{\ep \to 0}\frac{1}{\ep} \int_{\{t<u<t+\ep\}} \div\left(\frac{\nabla u}{|\nabla u|}\right)=\int_{\{u=t\}} \frac{1}{|\nabla u|}\div\left(\frac{\nabla u}{|\nabla u|}\right).
 \]
 Hence, the first equation in \eqref{L-der1} is proved. In order to see the second representation formula in~\eqref{L-der1} we directly compute that
 \begin{equation} \label{lem:L''-ident}
  \div\left(\frac{\nabla u}{|\nabla u|}\right)\frac{1}{|\nabla u|}=\frac{1}{|\nabla u|^3}(\Delta u |\nabla u|-\langle \nabla u, \nabla|\nabla u|\rangle).
 \end{equation}  
 From this the aforementioned formula follows immediately for harmonic function $u$. 
 
 In order to show assertions~\eqref{L'' n} and \eqref{L-der2} we begin with the following computations involving~\eqref{L-der1} and the Stokes  theorem:
 \begin{align*}
 L''(t)&=\lim_{\ep \to 0}\frac{L'(t+\ep)- L'(t)}{\ep} \\
 &=\lim_{\ep \to 0}\frac{1}{\ep}\left(\int_{\{u=t+\ep\}}   \left \langle -\frac{\nabla u}{|\nabla u|}, \frac{ \nabla |\nabla u|}{|\nabla u|^2} \right \rangle- \int_{\{u=t\}}  \left \langle -\frac{\nabla u}{|\nabla u|}, \frac{ \nabla |\nabla u|}{|\nabla u|^2} \right \rangle \right)\\
 &= \lim_{\ep \to 0}\frac{1}{\ep} \int_{\{t<u<t+\ep\}} -\div\left(\frac{\nabla |\nabla u|}{|\nabla u|^2}\right).
   \end{align*} 
 We continue our computations as in the proof of~\eqref{L-der1} and employ the Bochner identity~\eqref{formula-Bochner} together with~\eqref{cor-Bochner} to obtain that:
 \begin{align}	
  & \lim_{\ep \to 0}\frac{1}{\ep} \int_{\{t<u<t+\ep\}} -\div\left(\frac{\nabla |\nabla u|}{|\nabla u|^2}\right) \nonumber \\
  &= \lim_{\ep \to 0}\frac{1}{\ep} \int_{\{t<u<t+\ep\}} \left[-\frac{\Delta |\nabla u|}{|\nabla u|^2} + 2 \frac{\left|\nabla |\nabla u|\right|^2}{|\nabla u|^3}\right] \nonumber \\
  &=\lim_{\ep \to 0}\frac{1}{\ep} \int_{\{t<u<t+\ep\}} \left[-\frac{|\nabla^2 u|^2-\left|\nabla |\nabla u|\right|^2}{|\nabla u|^3} - \frac{\Ric(\nabla u, \nabla u)}{|\nabla u|^3}+ 2 \frac{\left|\nabla |\nabla u|\right|^2}{|\nabla u|^3}\right] \nonumber \\
  &=\lim_{\ep \to 0}\frac{1}{\ep} \int_{\{t<u<t+\ep\}} \left[-\frac{|\nabla^2 u|^2-2\left|\nabla |\nabla u|\right|^2}{|\nabla u|^3} - \frac{\Ric(\nabla u, \nabla u)}{|\nabla u|^3}+ \frac{\left|\nabla |\nabla u|\right|^2}{|\nabla u|^3}\right]. \label{L-der2-est}
 \end{align}
By applying the $n$-dimensional refined Kato inequality at \eqref{L-der2-est} we get
\[
-\frac{|\nabla^2 u|^2}{|\nabla u|^3} - \frac{\Ric(\nabla u, \nabla u)}{|\nabla u|^3}+ 3 \frac{\left|\nabla |\nabla u|\right|^2}{|\nabla u|^3}\leq (2-\frac{1}{n-1}) \frac{\left|\nabla |\nabla u|\right|^2}{|\nabla u|^3} - \frac{\Ric(\nabla u, \nabla u)}{|\nabla u|^3}.
\]
Upon applying this inequality in \eqref{L-der2-est} and using the coarea formula, conclusion \eqref{L'' n} follows readily, cf. the discussion following derivation of~\eqref{conv-level-sets}.

Now assume that $n=2$, so that $\Ric(\nabla u, \nabla u)=K|\nabla u|^2$. By applying the $2$-dimensional refined Kato identity~\eqref{refined-Kato2} we find that
 \[
  L''(t)= \lim_{\ep \to 0}\frac{1}{\ep} \int_{\{t<u<t+\ep\}} \frac{\left|\nabla |\nabla u|\right|^2}{|\nabla u|^3}- \frac{K}{|\nabla u|}.
 \]
 Using again the coarea formula we arrive at the estimate~\eqref{L-der2}.
 \end{proof}

We are now in a position to present one of the main results of this section and of the whole paper. Namely, in Theorems~\ref{thm-main1} and~\ref{thm-main2} we provide a characterization of annular domains with Gaussian curvature $K\leq 0$ in terms of the log-convexity of level sets for harmonic functions. The corresponding results for non-smooth surfaces with bounded integral curvature are discussed in Section 3.

The following theorem gives a counterpart of Alessandrini's result~\cite[Theorem 1.1]{al2} for harmonic functions on non-positively curved Riemannian $2$-manifolds. Recall that we do not assume that the curvature is constant, i.e. we present the argument for $K=K(x)$ for $x\in M^2$. Notice that, following \cite{al2,long}, the second part of the assertion justifies the name isoperimetric inequality. 
\begin{theorem}[Isoperimetric inequality]\label{thm-main1}
	
Let $\Omega$ be a $C^{1,\alpha}$-topological annular domain in a $2$-di\-men\-sional Riemannian manifold $(M^2, g)$ of non-positive curvature $K|_\Om \leq 0$. Let $t_1,t_2\in \R$ be such that $t_1<t_2$ and let us consider a continuous up to the boundary harmonic solution $u$ of the Dirichlet problem~\eqref{DP} in $\Om$.
Then  
\[
(\ln L(t))''\geq 0\quad\hbox{for all } t\in (t_1,t_2).
\]
Moreover, the equality in the assertion holds on $(t_1,t_2)$ if and only if $K\equiv 0$, in which case all the level curves of $u$ are concentric circles and $\Om$ is the regular (circular) annulus in the plane.
\end{theorem}

\begin{rem}\label{rem-c1a-reg}
 The $C^{1,\alpha}$-regularity assumption on the boundaries of topological annuli in subject is a consequence of the interior ball condition assumed in the Hopf lemma used in the proof of Theorem~\ref{thm-main1}, cf.~\cite[3.71]{Aubin} and the discussion in~\cite[Chapter 3.2]{gt}.  It is known that the $C^{1,\alpha}$-regularity characterizes domains with both interior- and exterior- ball conditions. Notice further that \cite{al2} assumes the $C^{2,\alpha}$-regularity and \cite{mz} the $C^{\infty}$-regularity of the boundary. However, in the general setting of BIC surfaces, see Theorem~\ref{thm-reshetnyak} below, we may weaken the regularity assumption and prove Theorem \ref{thm-main1} assuming that $\Gamma_i$ for $i=1,2$ are merely Jordan curves (see Remark \ref{r: Jordan bdy}). In this section we decided to assume higher regularity on boundary components of annuli for two reasons: (1) it gives a simpler proof than the one in Theorem~\ref{thm-reshetnyak} based on geometric arguments and so it may be more appealing to the reader not interested in the BIC setting; (2) the $C^{1, \alpha}$-regularity assumption would be more natural in higher dimensions, and with different operators such as the $p$-Laplacian, and thus allows to approach generalizations of Theorem~\ref{thm-main1}.
 
\end{rem}

Before presenting the proof of Theorem~\ref{thm-main1} we would like to comment that an important part of its proof, as in Alessandrini's result~\cite[Theorem 2.1]{al2}, is showing that the gradient of harmonic function in subject does not vanish. In the manifold setting of space forms a similar observation is proven in~\cite[Prop 3.2]{mz}, but it relies on some other auxiliary results. Moreover, the proof of~\cite[Theorem 1.1]{al2} could be obtained in our case by the analytic approach via the uniformization result for annuli (see Lemma~\ref{unif-annulus}) and harmonic functions in the plane, see Appendix A. The proof below is self-contained and refers to simple geometric arguments instead.

\begin{proof}
 Let us notice that by the maximum and minimum principles for the harmonic function $u$ we have that $\max_{\overline{\Om}} u=t_2$ and 
 $\min_{\overline{\Om}} u=t_1$. Moreover, by the Hopf lemma one gets that $|\nabla u|\geq const>0$ on $\partial \Om$, see~\cite[3.71, pg. 96]{Aubin}.  
 
 We first show that  $|\nabla u|\geq const>0$ on $\Omega$. On the contrary, let us suppose that there exists $x_0\in \Om$ such that $\nabla u(x_0)=0$. Consider the corresponding level curve $\gamma=\{x\in \Om: u(x)=u(x_0)\}$. 
 
 \emph{Claim:} There exists at least two simple closed curves $\gamma'_i\subset \gamma$, $i=1,2$.
 
 \emph{Proof of the claim:} The discussion can be localized in a neighbourhood $U_{x_0}\subset\Omega$ of the point $x_0$. Therefore, we introduce isothermal coordinates $(x,y)$ induced by the conformal chart $\phi:U_{x_0}\to\R^2$  (cf. also the discussion in the proof of Lemma~\ref{lem:log}). Since the harmonicity of $u$ is a conformal invariant in dimension $2$, we have that $\Delta_0 u=\Delta u=0$, where $\Delta_0$ denotes the planar Laplacian in the coordinates $(x,y)$. Since $\phi(x_0)$ remains a critical point for $u\circ \phi^{-1}$, we know by the theory of planar harmonic functions that the level curve in the neighbourhood of $\phi(x_0)$ forms a finite family consisting of at least two arcs intersecting at $\phi(x_0)$.
 This follows from the analyticity of harmonic functions in $\R^2$: indeed, the Taylor expansion of a harmonic function in the neighbourhood of a critical point starts with terms of order determined by number of derivatives vanishing at this point, see~\cite{hawi} and~\cite[Section 2.1]{dur}. 
Thus, for points on the level curve $\gamma$ we obtain, via $\phi^{-1}$, that there are at least two curves passing through $x_0$ contained in $\gamma$. If any of those branches would intersect $\partial \Om$, then by the assumption of continuity of $u$ up to the boundary, it would hold that $u(x_0)=t_1$ (or $u(x_0)=t_2$), hence the maximum of $u$ (or, respectively, minimum of $u$) would be attained in the interior of $\Om$, forcing $u=const$ by the strong maximum (respectively, minimum) principle. This is impossible, since $t_1\not =t_2$.
 
 Next, we rule out the possibility that the level curve $\gamma$ terminates at a point inside $\Om$. Indeed, suppose that there exists $y_0\in \gamma \cap \Om$, where $\gamma$ terminates, and consider two cases.

If $\nabla u(y_0)\not=0$, then the implicit function theorem implies that $\gamma$ can not terminate inside $\Om$, since it must be at least $C^1$ in a neighbourhood of $y_0$.

If $\nabla u(y_0)=0$, then by the discussion above, $\gamma$ would branch at $y_0$, contradicting assumption that it terminates there. 

To summarize, since $\gamma$ does not intersect $\partial \Om$ and does not terminate in $\Om$, it must contain at least two simple closed curve, denoted $\gamma'_i$, $i=1,2$, obtained by gluing regular curves (contained in $\gamma$). This ends the proof of the claim.


Since none of the curves $\gamma'_1$ and $\gamma'_2$ touches the boundary of $\Omega$, a topological argument together with the maximum principle allow us to infer that at least one of them bounds a domain $\Om'\Subset \Om$. Hence, we conclude that $\nabla u\not=0$ in $\Om$. 
%

We can thus apply Lemma~\ref{lem:log} obtaining that $\Delta (\log|\nabla u|)\leq 0$, and so the minimum principle together with the Hopf lemma implies that
$ \min_{\Om} |\nabla u|\ge \min _{\partial \Om} |\nabla u|>0$. 

We are in a position to complete the proof of the first part of the theorem.  By Lemma~\ref{lem: L''} and the Cauchy--Schwarz inequality we have that
\begin{align}
(L'(t))^2&\leq L(t) \int_{\{x\in \Om\,:\,u(x)=t\}} \left|\left \langle -\frac{\nabla u}{|\nabla u|}, \frac{ \nabla |\nabla u|}{|\nabla u|^2} \right \rangle\right|^2 {\rm d}\mathcal{H}^1 \nonumber \\
	 &\leq L(t) \int_{\{x\in \Om\,:\,u(x)=t\}} \frac{ |\nabla |\nabla u||^2}{|\nabla u|^4} {\rm d}\mathcal{H}^1 \label{step-main-thm1} \\
	 &\leq L(t)L''(t). \nonumber
\end{align}
	 In order to show the second part of the assertion, suppose that $(\ln L(t))''=0$. This is equivalent to $LL''=(L')^2$ which then by the Cauchy--Schwarz and the H\"older inequalities reads
\begin{align*}
L(t) \left(\int_{\{x\in \Om\,:\,u(x)=t\}} \frac{\big|\nabla |\nabla u|\big|^2}{|\nabla u|^4} - \frac{K}{|\nabla u|^2} {\rm d}\mathcal{H}^1\right)&=\left(\int_{\{x\in \Om\,:\,u(x)=t\}} \left \langle -\frac{\nabla u}{|\nabla u|}, \frac{ \nabla |\nabla u|}{|\nabla u|^2} \right \rangle {\rm d}\mathcal{H}^1\right)^2\\
& \leq L(t)\left(\int_{\{x\in \Om\,:\,u(x)=t\}} \frac{ |\nabla |\nabla u||^2}{|\nabla u|^4}{\rm d}\mathcal{H}^1\right)
\end{align*}
Since $K\leq 0$, this inequality may hold only when $K\equiv 0$ in which case we reduce the discussion to the planar case and so Theorem 3.1 in~\cite{long} gives the second assertion of the theorem (see also \cite[Theorem 1.1]{al2}). 
\end{proof}

Next we prove that the assertion of Theorem~\ref{thm-main1} in fact characterizes open surfaces with negative curvature.

\begin{theorem}\label{thm-main2}
	
Let $(M^2,g)$ be an open surface. Then $K\leq 0$ in $M^2$ if and only if for every topological annulus domain $\Omega$ in $M$ with $C^{1,\alpha}$-boundary, and every real constants $t_1<t_2$ the following property is satisfied: 
 given the continuous up to the boundary harmonic solution $u$ of the Dirichlet problem~\eqref{DP} 
it holds $(\ln L(t))''\geq 0$ for all $t\in (t_1, t_2)$.
\end{theorem}
 
 The $C^{1,\alpha}$-regularity assumption on boundary of annuli follows from the analogous assumption in Theorem~\ref{thm-main1}, see Remark~\ref{rem-c1a-reg} following the statement of that theorem.

\begin{proof}
The sufficiency of the condition $K\leq 0$ is proven in Theorem~\ref{thm-main1}. In order to show its necessity, we will argue by contradiction and suppose that there exists a point $x_0\in M$ where $K(x_0)>0$. This will imply that there exist a harmonic function and an annular domain $\Om$, determined at the end of this proof, with $(\ln L(t))''< 0$.

By the Korn-Lichtenstein theorem, we can choose an isothermal chart $(U,\phi)$ where $U$ is an open set of $M$ containing $x_0$ and $\phi:U\to V:=\phi(U)\subset \R^2$ is a diffeomorphism such that $\phi(x_0)=0$. The metric $g$ pulled back to this coordinate system writes $(\phi^{-1})^\ast g(y)=\lambda^2(y)dy^2$ for some smooth conformal factor $\lambda:V\to (0,\infty)$. Moreover, up to a composition with a linear isometry of $\R^2$, we can suppose that $\lambda(0)=1$ and $\nabla\lambda(0)=0$. In particular, 
\begin{equation*}
\lambda(y)=1+\sum_{i,j=1}^2A_{ij}y_iy_j+o(|y|^2),
\end{equation*}
where $2A$ is the Hessian matrix $\nabla^2\lambda (0)$ written in this coordinate system. 
For later purposes, we compute 
\[
\partial_r\lambda(y)=\nabla\lambda (y) \cdot \frac y{|y|} = \frac{2}{|y|} y^TAy + o(|y|) = \nabla^2 \lambda(0) \left(y,\frac y{|y|}\right)+ o(|y|),
\]
where $\displaystyle \lim_{y\to 0}\frac {o(|y|)}{|y|}=0$. Similarly,
\[
\partial^2_{rr}\lambda(y)=\nabla(\partial_r\lambda(y)) \cdot \frac y{|y|} = \frac{2}{|y|^2} y^TAy + o(1)= \nabla^2 \lambda(0) \left(\frac y{|y|},\frac y{|y|}\right)+ o(1).
\]
Let us consider on $V\setminus \{0\}$ the (Euclidean) harmonic function $-\ln(|y|)$, and observe that $x\mapsto -\ln(|\phi(x)|)$ is harmonic on $(U,g)$ by the conformal invariance of the harmonicity on surfaces.
Using polar coordinates in $\R^2$, we express
\[\begin{aligned}
L(t)&=\int_{\{u=t\}}\lambda\,d\mathcal H^1=\int_{\partial B_{e^{-t}}(0)} \lambda\,d\mathcal H^1=
e^{-t}\int_{0}^{2\pi}\lambda (e^{-t},\theta) \,d\theta,\\
L'(t)&=-e^{-t}\int_{0}^{2\pi}\lambda (e^{-t},\theta) \,d\theta - e^{-2t}\int_{0}^{2\pi}\partial_r\lambda (e^{-t},\theta)\, d\theta,\\
L''(t)&=e^{-t}\int_{0}^{2\pi}\lambda (e^{-t},\theta) \,d\theta +3 e^{-2t}\int_{0}^{2\pi}\partial_r\lambda (e^{-t},\theta)\, d\theta +  e^{-3t}\int_{0}^{2\pi}\partial^2_{rr}\lambda (e^{-t},\theta)\, d\theta.
\end{aligned}
\] 
By the asymptotic behaviour of $\lambda$, we compute
\[
\begin{aligned}
\int_{0}^{2\pi}\lambda (e^{-t},\theta) \,d\theta &= 2\pi+\frac12e^{-2t}\int_{0}^{2\pi}  \nabla^2 \lambda(0)(\theta,\theta) \,d\theta + o(e^{-2t})\\
\int_{0}^{2\pi}\partial_r\lambda (e^{-t},\theta)\, d\theta &= e^{-t}\int_{0}^{2\pi}  \nabla^2 \lambda(0)(\theta,\theta) \,d\theta + o(e^{-t})\\
\int_{0}^{2\pi}\partial^2_{rr}\lambda (e^{-t},\theta)\, d\theta &=  \int_{0}^{2\pi}  \nabla^2 \lambda(0)(\theta,\theta) \,d\theta + o(1)
\end{aligned}\]
as $t\to\infty$. In the equalities above, with an abuse of notation we write $\theta$ for a point of polar coordinates $(1,\theta)$ when $\theta$ appears as an argument of $\nabla^2\lambda$. By inserting the latter equations in the previous relations, we obtain that
\begin{equation}\label{eq-main-thm2}
(L(t))^2(\ln L(t))'' = L(t)L''(t)- (L'(t))^2 = 4\pi e^{-4t}\int_{0}^{2\pi}  \nabla^2 \lambda(0)(\theta,\theta) \,d\theta + o(e^{-4t}).
\end{equation}
Now, we note the following facts:
\begin{itemize}
	\item by a direct computation, $\int_{0}^{2\pi}  \nabla^2 \lambda(0)(\theta,\theta) \,d\theta = \pi\Delta\lambda (0)$,
	\item by the conformal change formula for the Gaussian curvature we have $\Delta(\ln\lambda)=-\lambda^2K$, so that $\Delta\lambda(0)=-K(x_0)$.
\end{itemize}
Recall that we assumed $K(x_0)>0$, and thus obtain $(\ln L(t_0))''\leq 0$ for $t_0$ large enough by~\eqref{eq-main-thm2}.
 Choosing $t_1<t_0<t_2$, $\Omega:= \phi^{-1}(B_{e^{-t_1}}(0)\setminus B_{e^{-t_2}}(0))$ and $u=-\ln(|\phi(x)|)$ gives the contradiction.  
\end{proof}

 Theorem~\ref{thm-main2} can be further refined if the Gauss curvature satisfies $K\leq \kappa\leq 0$. 
\begin{prop}\label{prop sharp K}
Let $(M^2,g)$ be an open surface. Then $K\leq \kappa \leq 0$ in $M^2$  
if and only if for every topological annulus domain $\Omega$ in $M$ with $C^{1,\alpha}$-boundary, and every real constants $t_1<t_2$ the following property is satisfied: 
 given the continuous up to the boundary harmonic solution $u$ of the Dirichlet problem~\eqref{DP} it holds that
 \begin{align}\label{est sharp K}
	(\ln L(t))''&\geq 
	 -\frac{\kappa }{L(t)}\int_{\{u=t\}}\frac{1}{|\nabla u|^2} d\mathcal H^1,\quad \hbox{for } t_1<t<t_2.
	\end{align}
\end{prop}
\begin{rem}
The above proposition is sharp at least in the case of surfaces with constant curvature. In Example~\ref{ex hyperb sf} below, we construct a harmonic function $u$ on a hyperbolic surface for which the equality holds in~\eqref{est sharp K}.
\end{rem}
\begin{proof}
For the sufficiency part of the proposition, notice that the assertion follows by repeating the steps of the proof for Theorem~\ref{thm-main1} and by using equation~\eqref{L-der2} in ~\eqref{step-main-thm1}.

 In order to show the necessity of the assertion let us follow the steps of the corresponding part of the proof for Theorem~\ref{thm-main2} and argue by contradiction. Then, equation~\eqref{eq-main-thm2} together with facts following it allows us to obtain that 
\[
L(t)L''(t)- (L'(t))^2 = - 4\pi^2 e^{-4t}K(x_0).
\]
On the other hand $|\nabla u| (e^{-t},\theta)= \lambda (e^{-t},\theta) e^{t}$, so that 
\[
\int_{u=t}\frac{1}{|\nabla u|^2}\, d\mathcal H^1 =
e^{-t} \int_{0}^{2\pi} \frac{e^{-2t}}{\lambda^2 (e^{-t},\theta)}\,d\theta=
 2\pi e^{-3t} + o(e^t) .
\]
Hence, 
\[
L(t)L''(t)- (L'(t))^2 \geq - \kappa L(t)\int_{u=t}\frac{1}{|\nabla u|^2}\, d\mathcal H^1
 =
-\kappa  4\pi^2 e^{-4t} + o(e^{-4t}) .
\]
This gives a contradiction if $K(x_0)>\kappa $ at some point $x_0\in M^2$.
\end{proof}

It turns out that the inequality in Theorem~\ref{thm-main1} can be quantified in the setting of surfaces with pinched curvature, provided that the harmonic function defined on the annular domain is the restriction of a globally defined function which is harmonic on the whole surface.

We first prove the following proposition and then discuss an example of class of surfaces and harmonic functions satisfying the assumptions of the proposition.
\begin{prop}\label{thm-main2-2}
Let $(M^2,g)$ be a complete non-compact surface whose Gauss curvature satisfies 
\begin{equation}\label{assm-thm-main2-2}
-\kappa _1\leq K \leq -\kappa_2 \leq 0,
\end{equation}
for some $\kappa _1\geq \kappa_2\ge 0$ and let $u>0$ be harmonic on $M^2$. Then for every topological annulus domain $\Omega$ in $M$ with $C^{1,\alpha}$-boundary, and such that $u$ takes constant values $0<t_1<t_2$ on the boundary components of $\Om$, it holds that 
\[
 (\ln L(t))''\geq \frac{\kappa_2 }{\kappa _1}\frac{1}{t^2},\quad \hbox{for } t_1<t<t_2.
\]
\end{prop}
\begin{proof} Let $u>0$ be harmonic in $M^2$ satisfying~\eqref{assm-thm-main2-2} and $\Om\Subset M^2$ be an annulus such that $u$ takes constant values, respectively $t_1$ and $t_2$, on the boundary components of $\Om$, cf.~\eqref{def-top-ring}. By Proposition \ref{prop sharp K}
\begin{align}\label{com-est}
	(L'(t))^2&\leq 
	 L(t)\left(L''(t)-\kappa_2 \int_{u=t}\frac{1}{|\nabla u|^2} d\mathcal H^1\right).
	\end{align}
Then, the second claim of Thm 1.1 in~\cite{pli} for $M^2$ and $\lambda=0$ and $m=n=2$ asserts that
$|\nabla u|^2\leq \kappa _1 |u|^2$ for all points in $M^2$. Therefore,
\[
 -\kappa_2\int_{u=t}\frac{1}{|\nabla u|^2} d\mathcal H^1\leq -\frac{\kappa_2 }{\kappa _1}\int_{u=t}\frac{1}{|u|^2} d\mathcal H^1=-\frac{\kappa_2 }{\kappa _1}\frac{1}{t^2}L(t).
\]
This combined with inequality~\eqref{com-est} gives the assertion.
\end{proof}

\begin{ex}\label{ex hyperb sf}
In order to construct an example of $2$-manifolds satisfying the assumptions of Proposition~\ref{thm-main2-2}, let us consider the $2$-dimensional infinite cylinder endowed with a warped Riemannian metric defined by 
	\[
	(S, g)=(\R\times\mathbb{S}^1, \ud t^2+\left(\frac{\ln\lambda}{2\pi}\cosh t\right)^2
	\ud \theta^2).
\]	
	By standard computations, the Gaussian curvature of $S$ equals $K=-\frac{(\cosh t)''}{\cosh t}=-1$, so that $(S,g)$ is a hyperbolic surface. Indeed, it is isometric to the quotient of the
	hyperbolic plane $H^2:=(\R\times (0,\infty), \frac{dx^2+dy^2}{y})$ by the cyclic group of isometries generated  by the isometry $f:H^2\to H^2$ defined as $f(x):=\lambda x$ for a given $\lambda>1$.  
	A positive non-constant harmonic function $u:S\to (0,\pi)$ can be explicitly defined as $u(t,\theta):=2\arctan(e^t)$. Indeed, upon computing directly the Laplace--Beltrami operator $\Delta_Su$ we find that 
	$$
	\Delta_Su=\left(\partial^2_{tt}+\frac{(\cosh t)'}{\cosh t}\partial_t+\frac{4\pi^2}{\ln^2\lambda\cosh^2t}\partial^2_{\theta \theta}\right)u=0.
	$$
	Notice that for annular domains in $S$ defined as $[a,b]\times \mathbb{S}^1$ with $-\infty<a<b<\infty$, $u$ is constant at each boundary component of such an annulus.
	Finally we observe that $L(s)=\ln\lambda\cosh (u^{-1}(s))=\frac{\ln\lambda}{\sin s}$ and so it holds that
$$
	(\ln L(s))''=\frac{1}{\sin^2s}=\frac{1}{L(s)}\int_{\{u=s\}}\frac{1}{|\nabla u|^2} d\mathcal H^1\geq 1.
	$$
Notice that this also proves the sharpness of Proposition~\ref{prop sharp K}.
	
	More generally, the existence of a non-trivial positive harmonic function can be proved by using a theory introduced by P. Li and L.-F. Tam. 
		By Theorem 2.1 in~\cite{li-tam} a surface $S$ supports a non-constant global positive harmonic function provided that $S$ has at least two ends, and at least one of them is non-parabolic (cf. Definitions 0.3-0.5 in~\cite{li-tam}). In the example above, as we will now argue, $S$ has exactly two non-parabolic ends. Indeed, the set $S\setminus B^S_R((0,0))$, for large enough $R$, consists of two components. Moreover, one can estimate that ${\rm Vol}(B^S_R((0,0)))\gtrsim \ln\lambda e^R+const$ and therefore
	$$
	\int^{\infty} \frac{R}{{\rm Vol}(B^S_R((0,1)))}\,\ud R<\infty,
	$$
	which implies that both ends are non-parabolic (see ~\cite{holop} for the details).
	This in particular suggests that the example above is stable under small perturbations of the metric.
\end{ex}

\section{Surfaces with bounded integral curvature}

In this section we focus our attention on surfaces with bounded integral curvature (\textsl{BIC surfaces} in short), also known in the literature as Alexandrov surfaces. This class represents a very important example of non-smooth (singular) spaces. Indeed, among BIC surfaces we find, for instance, the polyhedral surfaces (both Euclidean, hyperbolic, spherical or more generally Riemannian), the metric surfaces with curvature either lower or upper bounded in the sense of Alexandrov, and thus in particular the Gromov--Hausdorff limits of surfaces with lower bounded Gaussian curvature and $RCD(k,2)$ surfaces. Moreover, the BIC surfaces have a number of nice structure properties which permit to develop a differential calculus, including a singular (in some sense) Riemannian metric. We refer to~\cite{re, tr-AnnIHP, tr} for comprehensive introductions to the theory of Alexandrov surfaces, see also the more recent ~\cite{ab, fi, klp} and references therein. A significant class of BIC surfaces consists of surfaces with conical singularities, i.e.~surfaces which are smooth Riemannian except for a discrete set of singular points where the metric structure is at an infinitesimal level the one of a Euclidean cone; see for instance~\cite{bdm, ch, de, tr2} and their references. The results of this section apply to general Alexandrov surfaces with non-positive curvature measure (see details below), including surfaces with conical singularities whose angles at the vertices are greater than $2\pi$, and surfaces of $CAT(0)$ type. 
 
We start with recalling the basic definition of a \emph{surface of bounded integral curvature}. Let $(S,d)$ be a compact topological surface $S$ endowed with a metric function $d:S\times S \to [0,+\infty)$ which induces the topology of $S$ and is geodesic, i.e. $(S,d)$ is a geodesic metric space.
The surface $(S,d)$ is said to have \emph{bounded integral curvature}, if there exists a sequence of smooth Riemannian metrics $(g_j)$ on $S$ such that all $\int_S |K_{g_j}|\,d\mu_j$ are uniformly bounded and $d_j\to d$ uniformly in $j$. Here $K_{g_j}$ and $\mu_j$ denote, respectively, the Gaussian curvature and the Riemannian area measure of $(S, g_j)$, and $d_j$ is the distance induced by $g_j$ on $S$.
The \textsl{a posteriori} unique weak limit of the sequence of measures $(K_{g_j}\mu_j)$, denoted by $\omega$, is called the \emph{curvature measure of $(S,d)$}.
There exist several non-trivially equivalent definitions of BIC surface which are based for instance on polyhedral approximations (instead of Riemannian ones), or on a direct definition of $\omega$ as an outer measure exploiting the Gauss-Bonnet theorem at the infinitesimal level. Here, we are in particularly interested in the following analytic characterization.

Following \cite{tr}, let $S$ be a closed surface (i.e. compact surface without boundary) and $V (S, h)$ be the space of functions $u:S\to\R$ such that $\mu=\Delta_hu$ is a measure (note that for us a measure is a signed measure, i.e. not necessarily positive). For every $v\in V (S, h)$ and every $x,y \in S$, we set
\[
d_{h,v}(x, y) = \inf\left\{  \int_0^1 e^{v(\alpha(t))}\sqrt{h(\dot\alpha(t),\dot\alpha(t))}	\,dt\right\},
\]
where the infimum is taken over all the Lipschitz paths $\alpha: [0,1] \to S$ such that $\alpha(0) = x$  and $\alpha(1) = y$. It turns out that $d_{h,v}$ is indeed a distance \cite[Proposition 5.3]{tr}. In analogy with the smooth case, by using a common notation, we say that the distance $d$ on $S$ is induced by the (singular) Riemannian metric $e^{2v}h$; see also \cite{ab} for results about the regularity of this latter Riemannian metric. Since the harmonicity of a function is a conformal invariant in dimension two, the above relation permits to define harmonic functions on open subsets of $(S,d)$ as the functions which are harmonic with respect to the metric $h$. More precisely, we define the Laplace-Beltrami operator of $(S,d)$ by 
\[
\Delta_S:=\Delta_{(S,d)}:=e^{-2v}\Delta_h.
\]

Recall that a surface $(S,d)$ with bounded integral curvature is said to be without cusps, if $\omega(\{x\})< 2\pi$  for all $x\in S$. In particular, both the $BIC$ surfaces with non-positive curvature measure and all the smooth surfaces are automatically without cusps. Therefore, surfaces studied in Section 2 under the assumption $K\leq 0$ are covered by our discussion here as well.

We have the following theorem due to \cite{re-Sib} (see also~\cite[Theorem 7.1]{tr} and \cite{hu}).
\begin{theorem}\label{BIC-riem}
	Let $(S,d)$ be a closed surface with bounded integral curvature without cusps. Then, there exist a smooth Riemannian metric $h$ on $S$ and a function $v\in V(S,h)$ such that $d=d_{h,v}$. 
\end{theorem}

Here we obtain the following generalization of Theorem~\ref{thm-main1} to the non-smooth case.

\begin{theorem}\label{thm-reshetnyak}
Let $(S,d)$ be a surface of bounded integral curvature $(S,d)$ with non-positive curvature measure $\omega$, and let $\Omega$ be a non-degenerate topological annulus domain in $(S,d)$ whose boundary components are the Jordan curves $\Gamma_i$, $i=1,2$. Let $t_1,t_2\in \R$ be such that $t_1<t_2$ and let us consider a continuous up to the boundary solution of the following Dirichlet problem in $\Om$ for the Laplace--Beltrami harmonic operator $\Delta_S$ on $S$:
	\[
	\begin{cases}
	\Delta_{S} u=0 & \hbox{in } \Om,\\
	u|_{\Gamma_1}=t_1, & u|_{\Gamma_2}=t_2.
	\end{cases}
	\]
Then $t\mapsto \log L(t)$ is a convex function on $[t_1,t_2]$.
\end{theorem}

\begin{rem}\label{r: Jordan bdy}
	As a special case of the Theorem~\ref{thm-reshetnyak}, we get a generalization of Theorems~\ref{thm-main1} and \ref{thm-main2} to non-degenerate annular domains in smooth Riemannian manifolds whose boundary consists of Jordan curves, i.e. no further smoothness of boundary curves is assumed. According to our knowledge this result is new also in the Euclidean case. 
	\end{rem}

\begin{rem} It is natural to conjecture that the converse of Theorem \ref{thm-reshetnyak} is also true, i.e.~that the function $L(t)$ is convex for every solution of the Dirichlet problem on topological annulus domain only if the BIC surface has non-positive curvature. In particular, this would generalize Theorem \ref{thm-main2} to the singular setting. This appear as a quite difficult question, due to the possible presence of a purely singular positive part of the curvature measure (i.e.,~supported on a set of Hausdorff dimension strictly smaller than $2$).
\end{rem}

In the proof of the theorem we need the following uniformization lemma for a non-degenerate topological annulus on a BIC surface. Recall that, by \textsl{non-degenerate} we mean an annulus not homeomorphic to a punctured disc.
\begin{lem}\label{unif-annulus}
	Let $(S,d)$ a compact surface with bounded integral curvature without cusp. Let $\Om\Subset S$ be a non-degenerate topological annulus in $(S,d)$. Then $(\Om, d_\Om)$ is isometric to the annulus $A_{1, R}:=\{1< |z|< R\}\subset \C$, for some $R>1$, endowed with a conformally flat metric. More precisely, there exists a homeomorphism $\phi:A_{1, R}\to \Omega$ and a function $\bar{v}\in V(A_{1, R}, g_0)$ such that 
	\[
	d_\Omega(\phi(x),\phi(y))=d_{g_0,\bar{v}}(x,y),\quad \hbox{for all }x,y\in A_{1,R},
	\] 
where $g_0$ stands for the Euclidean metric $g_0=\ud x ^2+\ud y^2$.
\end{lem}
Here $d_\Omega$ stands for the intrinsic metric induced by $d$ on $\Omega$ and, as above, $V(A_{1, R},g_0)$ is the space of functions $u:A_{1,R}\to\R$ such that $\mu=\Delta_hu$ is a (signed) measure. Moreover, for every $v\in V (A_{1, R},g_0)$ and every $x,y \in S$, we set
\[d_{g_0,v}(x, y) = \inf\left\{  \int_0^1 e^{v(\alpha(t))}\sqrt{g_0(\dot\alpha(t),\dot\alpha(t))}	\,dt\right\},
\]
where the infimum is among all the Lipschitz paths $\alpha$ in $A_{1,R}$ connecting $x$ to $y$.


\begin{proof}[Proof of Lemma~\ref{unif-annulus}]
By Theorem \ref{thm-reshetnyak}, there exist a smooth Riemannian metric $h$ on $S$ and a function $v\in V(S,h)$ such that $d=d_{h,v}$. Since $\Omega$ is non-degenerate, by approximation, we can find an open collared neighbourhood $\Omega'$ of $\Omega$ such that $\Omega\Subset \Omega'\subset S$ and $\partial \Omega'$ is smooth in $(S,h)$. Then, there exists a smooth Riemannian metric $\tilde h$ on $\mathbb S^2$ such that $(\Omega',h)$ embeds isometrically in $(\mathbb S^2,\tilde h)$, see e.g. Theorem A in \cite{pv}.
According to the uniformization theorem for surfaces, there exists a $\psi\in C^\infty(\mathbb S^2)$ such that $\tilde h=  e^{2\psi}g_1$, $g_1$ being the round metric on $\mathbb S^2$ of constant Gaussian curvature $1$. Hence, up to isometries, $(\Omega,d_\Omega)$ can be identified with a non-degenerate topological annular domain $\Omega_{\mathbb S^2}$ of $\mathbb S^2$ endowed with the metric $e^{2(\psi+v)}g_1$.
Finally, let $P\in \mathbb S^2\setminus \overline{\Omega_{\mathbb S^2}}$. The stereographic projection from $P$ induces a conformal diffeomorphism of $(\Omega_{\mathbb S^2},g_1)$ onto $(\Omega_{\C},g_0)$, for some   non-degenerate topological annular domain $\Omega_{\C}\Subset \C$. For later purposes, note that this conformal diffeomorphism can be defined on a slightly larger open set properly containing $\Omega_{\mathbb{S}^2}$ giving rise via the aforementioned stereographic projection to an open set $\hat\Omega_\C\Supset \Omega_\C$. In order to complete the proof, we apply a version of the Riemann mapping theorem to obtain that $\Omega_{\C}$ is conformal to $A_{1,R}\subset \C$ for some $R>1$, see Theorem 10 in \cite[Chapter 6]{ah}. All together, we have obtained the existence of an isometry $\phi:(A_{1,R},e^{2\bar{v}}g_0)\to (\Omega,d)$ for some conformally flat metric $e^{2\bar{v}}g_0$, where $\bar{v}:A_{1,R}\to \R$ is such that
$\Delta_{g_0}\bar{v}$ is a measure on $A_{1,R}$. Note that $\bar{v}$ is non-smooth in general, so that $e^{2\bar{v}}g_0$ is not Riemannian. However, it is a well-defined length metric in which the length
of any constant-speed curve $\alpha:[0.1]\to A_{1,R}$ is given by the formula $\int_0^1 e^{\bar{v}(\alpha(t))}\sqrt{g_0(\dot\alpha(t),\dot\alpha(t))}\,dt$.
\end{proof}


\begin{rem}
 We encourage readers less familiar with the BIC setting to see the presentation in the analytic proof of Theorem~\ref{thm-main1} in Appendix A, especially computations for $L$.
\end{rem}

\begin{proof}[Proof of Theorem~\ref{thm-reshetnyak}] 
	 By Lemma~\ref{unif-annulus}, we have the existence of an isometry $\phi$ from the annulus $A_{1,R}:=\{1<|z|<R\}\subset \mathbb C = \R^2$ endowed with the conformally flat metric $e^{2\bar{v}} g_0$ onto the topological annular domain $(\Omega, h)$. Here $g_0:=dx^2+dy^2$, $R>1$ and $\bar{v}:A_{1,R}\to\R$ is such that $\bar{v}\in V(A_{1,R},g_0)$, i.e. $\Delta_{g_0}\bar{v}$ is a measure. More precisely, $\bar v$ is obtained as $\bar v = \hat v\circ F$ where $F$ is a conformal univalent mapping from $A_{1,R}$ to $\Omega_\C$ and $\hat v\in V(\hat\Omega_\C,g_0)$; see the proof of Lemma~\ref{unif-annulus}.
	According to \cite[Corollary 6.3]{tr}
	, we know that the curvature measure $\omega$ of $(S,d)$ is pulled-back via the isometry $\phi$ precisely to the measure $(\phi)^\ast(\omega)=-e^{-2\bar v}\Delta_{g_0}\bar{v}$ on $A_{1,R}$. With an abuse of notation from now on, for the sake of simplicity, we will denote by $\omega$ the measure $(\phi)^\ast(\omega)$ defined on $A_{1,R}$.  
	By a similar reasoning, $(\hat\phi)^\ast(\omega)=-e^{-2\hat v}\Delta_{g_0}\hat{v}$ on $\hat \Omega_\C$, where $\hat\phi$ is the conformal map from $\hat \Omega_\C$ to $(S,d)$. 

Let $u$ be the harmonic function in the statement of the theorem. Up to a vertical shift, we can assume without loss of generality that $t_2=0$. Then, up to a multiplicative constant, we can assume that  $t_1 = -\ln R$. Note that both of these transformations do not affect the conclusion of the theorem. Thus, $u$ is the unique harmonic function on $A_{1,R}$ with boundary data $0$ and $-\ln R$, namely $u=-\ln |z|$.

If the curvature measure $\omega$ of $(S,d)$ is non-positive, then $\bar{v}$ verifies 
\[
\Delta_0\bar{v} \geq 0,
\]
in the weak sense. More precisely, by definition of the distributional Laplacian \cite[Chapter 4.3]{ArGa}, one has
\[
\int_{A_{1,R}} \bar{v} \Delta_0\psi \,d\lambda \ge0
\]
for all $0\le \psi\in C^\infty_c(A_{1,R})$.
Furthermore, \cite[Theorem 4.3.10]{ArGa} implies that $\bar v$ is subharmonic (see also \cite[p. 99]{re}), meaning that $\bar{v}$ is upper semicontinuous in $A_{1,R}$,  $\bar{v}\not\equiv-\infty$ on $S$, and that $\bar{v}$ satisfies the subharmonic mean value property for each ball compactly contained in $A_{1,R}$, cf. \cite[Definition 3.1.2]{ArGa}. Similarly, $\hat v$ is subharmonic on $\hat \Omega_\C$, and it is there upper semicontinuous. 

We are going to prove that $L(t)$ is a convex function on $[t_1, 0]$. 
First, we \textsl{claim} that $L(t)$ is convex on any subinterval $I_\epsilon:=(t_1+\epsilon,-\epsilon)$, for $0<\epsilon<\frac12\ln R$. Note that $L(t)$ can be expressed in the integral form as follows
\begin{equation}\label{eq:lengthBIC}
L(t) = \int_{|z|=e^{-t}}e^{\bar v}\,d\sigma,
\end{equation}
see, for instance, \cite[Example 1.5]{de} and the corresponding computations for $L(t)$ in Appendix A.
By definition of $u$ it holds that $\Omega_{\epsilon/2}:=u^{-1}(I_{\epsilon/2})\Subset A_{1,R}$. 
Then, by \cite[Theorem 3.3.3]{ArGa}, there exists a decreasing sequence $(s_j)$ of subharmonic functions defined in $A_{1,R}$ and smooth in $\Omega_{\epsilon/2}$ for sufficiently large $j$, constructed by mollifications (cf. \cite[(3.3.1)]{ArGa}), such that $(s_j)$ converges pointwise to $\bar v$ on $\Omega_{\epsilon/2}$. We associate with $(s_j)$ a sequence $(g_j)$ of smooth Riemannian metrics on $\Omega_{\epsilon/2}$ conformally defined by $g_j := e^{2s_j}g_0$.  By the proof of~\cite[Theorem 3.3.3]{ArGa} we have that $s_j$ are subharmonic for large enough $j$ such that $1/j<\dist(\Om_\epsilon, \partial A_{1,R})\leq e^{-t_1}(e^{\epsilon/2}-1)$. Let us denote the smallest of such $j$ by $j_0$ and note that $j_0$ depends only on $\epsilon$. Therefore, for $j\geq j_0$ the smooth metrics $g_j$ have nonpositive  Gaussian curvatures $K_{g_j}$: 
\[
K_{g_j}=(-\Delta_{0}s_j)\, e^{-2s_j}\leq 0.
\]
Since the harmonicity is a conformal invariant in dimension $2$, we get that the harmonic function $u$ given by assumptions of Theorem~\ref{thm-reshetnyak} is harmonic in $\Om_\epsilon$ also with respect to metric $g_j$ for every $j\geq j_0$. Moreover, the annular domain $\Omega_\epsilon$ has smooth boundary in $(\Omega_{\epsilon/2},g_j)$. Applying Theorem~\ref{thm-main1}, we get that $t\mapsto L_j(t)$ is a convex function on $I_\epsilon$. Here,
\[
L_j(t) = \int_{|z|=e^{-t}}e^{s_j}\,d\sigma
\]
is the length of the level set $\{u=t\}$ with respect to the metric $g_j$. Observe that the sequence of functions $(e^{s_j})$ is monotone non-increasing, locally uniformly bounded above by $e^{s_{j_0}}$ and trivially bounded below by $0$, and converges pointwise to $e^{\bar v}$. 
Hence, by the monotone convergence theorem we get that $L_j(t)\to L(t)$ pointwise on $I_\epsilon$.  
To conclude the proof of the claim, we observe that the pointwise limit of a sequence of convex functions is convex as well. 
Since $\epsilon>0$ is arbitrary we straightforwardly deduce that $L(t)$ is convex on open interval $(t_1,0)$. 

It remains to prove the convexity of $L(t)$ at the boundary points $t_1$ and $0$, hence on the whole interval $[t_1,0]$. To this end, we note that $\bar v$ (so far defined on $A_{1,R}$) is indeed well-defined and upper semicontinuous on the closure $\overline{A_{1,R}}$. Indeed, $F$ can be continuously extended to a map $\bar F:\overline{A_{1,R}}\to\overline{\Omega_\C}$ by a version of the Osgood--Carath\'eodory Theorem, see \cite[Section 8]{OsTa} or \cite[Theorem 2.8]{GaSh}. Since, $\hat v$ is upper semicontinuous on $\hat\Omega_\C\supset \overline{\Omega_\C}$, we may define $\bar v := \hat v\circ \bar F$ on $\overline{A_{1,R}}$. The convexity at the boundary points is thus a simple consequence. Namely, for $t<0$ we have
\begin{align*}
L(0)&=\int_{|z|=1} e^{\bar v(z)}\,d\sigma 
\\&= \int_{|z|=1} e^{\bar v(z)}\,d\sigma -\int_{|z|=e^{-t}} e^{\bar v(z/|z|)}\,d\sigma + \int_{|z|=e^{-t}} e^{\bar v(z/|z|)}\,d\sigma - \int_{|z|=e^{-t}} e^{\bar v(z)}\,d\sigma +\int_{|z|=e^{-t}} e^{\bar v(z)}\,d\sigma \\
& = \int_{|z|=1} e^{\bar v(z)}\,d\sigma -\int_{|z|=1} e^{-t}e^{\bar v(z)}\,d\sigma +  \int_{|z|=e^{-t}} \left[e^{\bar v(z/|z|)}- e^{\bar v(z)}\right]\,d\sigma + L(t).
\end{align*}
Taking the $\limsup$ as $t\to 0^-$ gives $L(0)\ge \limsup_{t\to 0^-} L(t)$. Thus $L(t)$ is convex on $(t_1,0]$. A similar argument permits to deal with the other boundary point $t_1$.
\end{proof}
\begin{rem}\label{rmk:AG}
	Notice that, since the integrand in \eqref{eq:lengthBIC} is the exponential of a subharmonic function, one can deduce the log-convexity of $L$ in $I_\epsilon$  from \cite[Theorem 3.5.7(ii)]{ArGa}. However, for the readers convenience and in order to present a self-contained argument, we decided to provide the complete discussion in the proof of Theorem \ref{thm-reshetnyak}.
\end{rem}

\section{Curvature of level sets and PDEs}

The purpose of this section is to study some curvature and length estimates for the level curves of harmonic functions on Riemannian $2$-manifolds. Similar studies in the setting of planar harmonic functions and more general quasilinear equations have been conducted by e.g. Alessandrini~\cite{al} and Talenti~\cite{tal}. In the setting of manifolds, let us mention results by Ma--Qu--Zhang~\cite{moz}, Ma--Zhang~\cite{mz} and Wang--Wang~\cite{ww}. The key results in those papers are obtained assuming that the Gaussian curvature of the surface is constant, or on higher-dimensional space-forms . One of the main novelties of our work is to allow the curvature to vary. Also, we study more types of curvatures of level sets than in the aforementioned papers; see the discussion before the statement of the main result of this section, Theorem~\ref{thm-pdes}. In its corollaries we investigate the weak and strong maximum and minimum principles for curvatures of level curves (Corollaries~\ref{c: k}\,--\,\ref{c:cor3.7}) and further estimates for the length of level curves and its derivative (Corollary~\ref{c:cor3.8-2ndversion}).

 We begin the presentation of the results with a technical lemma, which provides the key identities for computations in the proof of Theorem~\ref{thm-pdes}. The second assertion of the following lemma generalizes~\cite[Lemma 2.3]{ww} to the setting of non-constant sectional curvature (see identities (2.3) and (2.10) in~\cite[Lemma 2.1]{cm} for the proof of the lemma).
 
%
\begin{lem}\label{lem-rr}
 Let $u:M^n\to \R$ be a smooth function on a Riemannian manifold $M^n$ (with possibly non-constant sectional curvature). Then
\begin{align}
 u_{ijk}&=u_{ikj}+u_l R_{lijk} \label{lem-rr1}\\
	\sum_k u_{ijkk}&=\sum_ k u_{kkij}-\sum_{l,m} 2u_{lm}R_{iljm}\nonumber\\
	&+\sum_l\left[\Ric_{il}u_{lj}+\Ric_{jl}u_{li}+u_l(\nabla_i\Ric)_{lj}+u_l(\nabla_j\Ric)_{li}-u_l(\nabla_l\Ric)_{ji}\right],
\label{lem-rr2}
%
\end{align} 
where $R_{ijkl}$ and $\Ric_{ij}$ stand, respectively, for the coefficients of the curvature tensor and of the Ricci tensor, while a function with subscript indices, e.g. in the expression $u_{ij}$, denote the (higher order) covariant derivative of the function with respect to a local orthonormal frame.
%
%
%
\end{lem}



Next, we generalize the main result of Wang--Wang~\cite[Theorem 1.3]{ww} and also Theorem 3 in Talenti~\cite{tal} in several different directions:
\begin{itemize}
\item[(1)] we study Riemannian manifolds with non-constant Gauss curvature $K=K(x)$ for $x\in M^2$, whereas in~\cite{ww} it is assumed that $K\equiv const$.
\item[(2)] While~\cite{ww} investigates only the case of \emph{the curvature of curves of the steepest descent} $h$ we study both $h$ and \emph{the curvature of the level curves} $k$:
\begin{equation}\label{eq-k-h}
 h:=\div\left(\frac{\star \nabla u}{|du|}\right),\qquad k:=-\div\left(\frac{\nabla u}{|\nabla u|}\right). 
\end{equation}
Here $\star$ denotes the Hodge star operator acting on tangent vectors. If we fix a local orthonormal frame $\{e_1,e_2\}$ and write $u_i=e_i(u)$, $i=1,2$, then $h$ can be written in the more friendly equivalent form
\[
h:=-\div\left(\frac{(u_2,-u_1)}{|\nabla u|}\right).
\]

\item[(3)] Equations~\eqref{pde1} and~\eqref{pde1-star} generalize the harmonicity result obtained for the planar case and $K=0$, see in~\cite[Theorem 3(i)]{tal}, while inequalities~\eqref{pde2} and \eqref{pde2-star} generalize the subharmonicity property for $h$ and $k$ in~\cite[Theorem 3(ii)]{tal}.

\item[(4)] The maximum and minimum principles for $\frac{k}{|\nabla u|}$ and $\frac{h}{|\nabla u|}$ presented in Corollaries~\ref{c: k} and~\ref{cor-up-lo-bound_2ndversion} below, generalize Corollary 1.1 in~\cite{ww} to the setting of surfaces with non-constant Gauss curvature.
\end{itemize}
 
\begin{theorem}\label{thm-pdes}
 Let $\Om\subset M^2$ be a domain in the Riemannian manifold with Gaussian curvature $K$ (not necessarily constant). Let further $u:\Om\to \R$ be a harmonic function with no critical points in $\Om$. Then the function $\phi=\frac{k}{|\nabla u|}$ satisfies the following differential equation:
\begin{equation}\label{pde1}
 \Delta \phi+2K\phi=\frac{\langle \nabla K,\nabla u\rangle}{|\nabla u|^2}.
\end{equation}
A similar equation holds true for $\frac{h}{|\nabla u|}$:
\begin{equation}\label{pde1-star}
 \Delta \left(\frac{h}{|\nabla u|}\right)+2K\frac{h}{|\nabla u|}=-\frac{\langle \nabla K,\star\nabla u\rangle}{|\nabla u|^2}.
\end{equation}
 Moreover, it holds that
\begin{equation}\label{pde2}
 -\Delta \ln |k| \geq K-\frac{1}{|k|}\langle \nabla K,\frac{\nabla u}{|\nabla u|}\rangle,\quad k\not=0.
\end{equation}
A similar inequality holds true for $|h|$ at points where $h\not=0$:
\begin{equation}\label{pde2-star}
	-\Delta \ln |h| \geq K+\frac{1}{|h|}\langle \nabla K,\frac{\star \nabla u}{|\nabla u|}\rangle,\quad h\not=0.
\end{equation}	
\end{theorem}

The proof of the above theorem is largely computational and, therefore, we present it in Appendix B. Here instead we focus on discussing some consequences and applications of Theorem~\ref{thm-pdes}. 

First, let us comment on conditions implying that a harmonic function has no critical points in the domain of its definition.
 As observed in the previous section,  a harmonic solution $u$ to the Dirichlet problem with constant boundary data on an annular domain satisfies $|\nabla u|>0$. Moreover, let $u$ solve the harmonic Dirichlet problem in a domain $\Om\subset M^2$ (not necessarily an annulus) and $u\in C^1(\overline{\Om})$. According to the discussion following Theorem 3.2 in~\cite{mag} in the planar Euclidean case (zero curvature), one can estimate the number of critical points of $u$ inside $\Om$ in terms of information given by prescribed tangential, normal, co-normal, partial, or radial derivatives of $u$. This is due to the flexibility of assumptions in \cite[Theorem 3.2]{mag}; we also refer to Theorems 2.1 and 2.2 in~\cite{almag} for further details. In particular, the appropriate Neumann condition may result in no critical points of $u$. For instance, consider the planar unit disc, a vector field $l=z/|z|$ and any boundary data for a harmonic Dirichlet problem such that $u_l$ has two zeros on the unit circle. Then, according to ~\cite[Theorem 3.2]{mag} (and with its notation), it holds that $M^+=M^-=1$ and by the argument principle for $l$ we have $D=1$. Thus, $u$ has no critical points inside the disc.   
 As commented in \cite[Section 3.2]{mag}, corresponding results can be obtained for harmonic functions defined on domains in Riemannian $2$-manifolds $M^2$. 


We proceed now to present some of the consequences of Theorem~\ref{thm-pdes}.


\begin{cor}\label{c: k}
Let $u$ be harmonic in a domain $\Om\subset M^2$ and $\Om'\Subset \Om$ be any subdomain such that $\overline{\Om'}$ does not contain critical points of $u$. If $K\geq 0$ and $\langle \nabla K,\nabla u\rangle\leq 0$ in $\Om'$, then $|k|$ attains its minimum on the boundary of $\Om'$ (provided that $k\neq 0$ in $\Omega$). 

Similarly, if $K\geq 0$ and $\langle \nabla K,\star \nabla u\rangle\geq 0$ in $\Om'$, then $|h|$ attains its minimum on the boundary of $\Om'$ (provided that $h\neq 0$ in $\Omega$).
\end{cor}

\begin{proof}
By assumptions of the corollary, we may infer from~\eqref{pde2} that 
\[
0\le K\leq -\Delta \ln |k| + \frac{\langle \nabla K,\nabla u\rangle}{|\nabla u|\,|k|}\leq  -\Delta \ln |k|.
\]
 Then, the minimum principle for the inequality $\Delta \ln |k| \leq 0$ implies that $|k|$ attains its minimum on the boundary of $\Om'$. The assertion involving $h$ is proved similarly.
\end{proof}

The minimum and the maximum principles for equation~\eqref{pde1} imply the following observation. Note that by assumptions of the corollary below, $|\nabla u|$ is bounded from below and above in $\overline{\Om'}$ for $\Om'\Subset \Om$, since $\Om$ is a domain and thus bounded by the definition. 
\begin{cor}\label{cor-up-lo-bound_2ndversion}
	 Let $u$ be harmonic in a domain $\Om\Subset M^2$ such that $\overline{\Om}$ does not contain critical points of $u$.

\begin{enumerate}
\item Let $K\le 0$ and $\langle\nabla K, \nabla u\rangle\!\ge\! 0$. If the maximum of $|\nabla u|^{-1}k$ is non-negative, then it is attained on $\partial \Omega$. 
	\item Let $K\le 0$ and $\langle\nabla K, \nabla u\rangle \le 0$. If the minimum of $|\nabla u|^{-1}k$ is non-positive, then it is attained on $\partial \Omega$.
		\item Let $K\ge 0$ and $\langle\nabla K, \nabla u\rangle \ge 0$. If the maximum of $|\nabla u|^{-1}k$ is non-positive, then it is attained on $\partial \Omega$. 
	\item Let $K\ge 0$ and $\langle\nabla K, \nabla u\rangle\!\le 0$. If the minimum of $|\nabla u|^{-1}k$ is non-negative, then it is attained on $\partial \Omega$.
\end{enumerate}	 

The same conclusions hold for $|\nabla u|^{-1}h$ up to replacing $\langle\nabla K, \nabla u\rangle$ with $\langle\nabla K, -\star \nabla u\rangle $ in the assumptions.	 
	\end{cor}

The above corollary together with Corollary~\ref{c: k} generalize~\cite[Corollary 1.1]{ww}, where the case of constant $K$ and only the curvatures $h$ and $\frac{h}{|\nabla u|}$ are considered.

\begin{ex}
 The class of surfaces and harmonic functions satisfying assumptions of Corollary~\ref{cor-up-lo-bound_2ndversion}~(1) contains, for instance, all conformally Euclidean annular domains, with the conformal factor satisfying $\Delta\varphi\geq 0 (\Leftrightarrow K\leq 0)$ and $\partial_r (e^{2\varphi}\Delta\varphi)\leq 0$. The latter condition is equivalent to $\langle \nabla u, \nabla K \rangle\geq 0$, as upon a conformal deformation and for a constant boundary data, it holds that $\nabla u$ is parallel to $\partial_r$ resulting in $\partial_r K=\partial_r (-e^{2\varphi}\Delta\varphi)$. Similar classes of examples can be constructed also for the assumptions of Corollary~\ref{cor-up-lo-bound_2ndversion}~(2), (3) and (4).
\end{ex}


Next, we obtain a minimum principle for the functions $k$ and $h$. 
\begin{cor}\label{c:cor3.7}
 Let $u$ be harmonic in a domain $\Om\Subset M^2$ such that $\overline{\Om}$ does not contain critical points of $u$. If nonconstant $k$ attains a minimum at a point $y$ in the interior of $\Omega$, then at $y$ one has necessarily $k\le |\nabla K|/K$.
Similarly, if nonconstant $h$ attains a minimum at a point $y$ in the interior of $\Omega$, then at $y$ one has necessarily $h\le |\nabla K|/K$.
%
%
 \end{cor}

Let us note that this corollary generalizes the second assertion in~\cite[Corollary 1.1]{ww} to the setting of non-constant Gaussian curvature and curvature $k$. In particular, our observation gives the minimum principle for $k$ for harmonic functions on sphere $\mathbb{S}^2$ not covered by~\cite{ww}.
\begin{proof}
 The Schwartz inequality implies that~\eqref{pde2} reads ($k>0$ by our assumptions):
 \[
  -\Delta \ln k +\frac{|\nabla K|}{k}\geq K.
 \]
We proceed by contradiction and suppose that $k>|\nabla K|/K$ in a neighbourhood of the minimum $y$. Set $v:=\frac{1}{k}>0$. Then $v \in C^{\infty}$ and upon substituting $v$ in the differential inequality above, one gets
 \begin{equation}
 0\leq \Delta v - \frac{|\nabla v|^2}{v}+|\nabla K|v^2 -Kv
\leq \Delta v \end{equation}
 Thus by the Hopf strong maximum, see~\cite[Theorem 3.5]{gt}, we conclude  that $k$ is constant, contradicting assumptions of the corollary. Hence, we get the assertion of the corollary for $k$. The reasoning for $h$ follows the same lines.
 \end{proof}
The maximum principle allows us to infer also the following upper bound for the growth of the function $L$.
\begin{cor}\label{c:cor3.8-2ndversion}
	Let $u$ be a harmonic solution to the Dirichlet problem~\eqref{DP} in a topological $C^{1,\alpha}$-annulus domain $\Om\subset M^2$ with constant boundary data $t_1$ and $t_2$, respectively.

If $K\leq 0$ and $\langle\nabla K,\nabla u\rangle\leq 0$ in $\Omega$, 
then for every level curve $\{u=t\}$ contained it the interior of $\Om$ it holds
\[
\left(\ln L(t) \right)' \leq  \max\{-\inf_{\partial \Om} (|\nabla u|^{-1}k);0\}.
\]
Similarly, if $K\geq 0$, $\langle\nabla K,\nabla u\rangle\leq 0$ and $k\ge 0$ in $\Omega$, 
then for every level curve $\{u=t\}$ contained it the interior of $\Om$ it holds
\[
\left(\ln L(t) \right)' \leq  -\inf_{\partial \Om'} (|\nabla u|^{-1}k).
\]
\end{cor}

\begin{proof}
Suppose first that the minimum of $|\nabla u|^{-1}k$ is non-positive. Then, by assertion (2) of Corollary~\eqref{cor-up-lo-bound_2ndversion} we have that $\inf_{\Om} |\nabla u|^{-1}k=\inf_{\partial \Om} |\nabla u|^{-1}k\leq 0$. Hence, by formula~\eqref{L-der1} for $L'$, we have
 \[
   L'(t)=\int_{\{x\in \Om\,:\,u(x)=t\}} \div\left(\frac{\nabla u}{|\nabla u|}\right)\frac{{\rm d}\mathcal{H}^1}{|\nabla u|}=
   \int_{\{x\in \Om\,:\,u(x)=t\}} \left(-\frac{k}{|\nabla u|}\right) {\rm d}\mathcal{H}^1\leq -\inf_{\partial \Om} (|\nabla u|^{-1}k) L(t).
 \]
If the minimum of $|\nabla u|^{-1}k$ is positive, then $\sup(-|\nabla u|^{-1}k)<0$, and so in the above estimate we get that $L'(t)<0$ for all $t_1<t<t_2$.

In order to prove the second assertion we apply part (4) of Corollary~\eqref{cor-up-lo-bound_2ndversion} and the similar reasoning as above.
\end{proof}

We finish this section by applying to our setting a maximum principle due to Mugnai-Pucci~\cite{mupu}. With respect to our previous results the conclusion is somehow weaker. However, here no sign assumptions on $K$ or $\nabla K$ are required.

\begin{cor}\label{cor-up-lo-bound2} 
 Let $u$ be harmonic in a domain $\Om\Subset M^2$ such that $\overline{\Om}$ does not contain critical points of $u$.
If $\|K\|_{L^{\infty}(\Om)}, \|\nabla K\|_{L^{\infty}(\Om)}<\infty$ and $|\nabla u|$ is bounded from below in $\Om$, then the following weak maximum principle holds for $|\nabla u|^{-1} k$:
\begin{equation}
\sup_{\Om} (|\nabla u|^{-1}k) \leq \left(1+c_1(|\Om|)\|K\|_{L^{\infty}(\Om)}\right) \sup_{\partial \Om} (|\nabla u|^{-1}k)+c_1(|\Om|)c_2,  \label{weak-max-k2}
\end{equation}
where $c_2=\|(|\nabla u|^{-1} k)^{+}\|_{L^2(\Om)}+\||\nabla K|\|_{L^{\infty}(\Om)}\|\frac{1}{|\nabla u|}\|_{L^{\infty}(\Om)}+\|K\|_{L^{\infty}(\Om)}^{\frac12}$.

Moreover, the corresponding weak maximum principles hold also for $|\nabla u|^{-1} h$ with estimate~\eqref{weak-max-k2} modified accordingly, i.e.
\[
\inf_{\Om} (|\nabla u|^{-1}k) \geq \left(1+c_1(|\Om|)\|K\|_{L^{\infty}(\Om)}\right) \inf_{\partial \Om} (|\nabla u|^{-1}k)-c_1(|\Om|)c_2.
\]

Similarly, the analogous weak minimum principles corresponding to estimate~\eqref{weak-max-k2} hold for $|\nabla u|^{-1} k$ and $|\nabla u|^{-1} h$ as well.
\end{cor}

In the planar setting, estimate \eqref{weak-max-k2} for $|\nabla u|^{-1}k$, as well as the corresponding one for $|\nabla u|^{-1}h$, hold in a stronger version without the constants due to harmonicity of these expressions, cf.~\cite[Theorem 3(i)]{tal} .

\begin{proof}
 Let us discuss the case of $\phi:=|\nabla u|^{-1} k$, keeping in mind that the discussion for $\phi:=|\nabla u|^{-1} h$ is analogous.
 By \cite[Theorem 3.1]{mupu}, and following the notation in~\cite{mupu}, we have that ${\bf A}(x, \phi,\nabla \phi)=A(\nabla \phi):=\nabla \phi$ and by smoothness of $\phi$ it holds that $A(\nabla \phi)\in L_{loc}^p(\Om, T\Om)$ for any $p>1$, hence in particular $\phi$ is $2$-regular. Moreover, 
 \[
 \langle A(\zeta), \zeta \rangle = |\zeta|^2 \quad \hbox{ and }\quad 
 B(x,z,\zeta):=2K(x)z+\frac{|\nabla K(x)|}{|\nabla u(x)|}\leq b_2|z|+b,
 \]
 with $b_2, b$ defined accordingly in terms of the $L^{\infty}$-norms of $K, \nabla K$ and $|\nabla u|^{-1}$. Then, the assertion of \cite[Theorem 3.1]{mupu} says that
\begin{align*}
\sup_{\Om} |\nabla u|^{-1}k &\leq \left(1+C(|\Om|)\|K\|_{L^{\infty}(\Om)}\right) \sup_{\partial \Om} |\nabla u|^{-1}k \\
&\,\,+C(|\Om|)\left(\|(|\nabla u|^{-1} k)^{+}\|_{L^2(\Om)}+\||\nabla K|\|_{L^{\infty}(\Om)}\|\frac{1}{|\nabla u|}\|_{L^{\infty}(\Om)}+\|K\|_{L^{\infty}(\Om)}^{\frac12}\right). 
\end{align*}
For the weak minimum principle for $\phi=|\nabla u|^{-1}k$ observe that $-\phi$ satisfies equation~\eqref{pde1} with the negative right-hand side. We apply the maximum principle to $-\phi$ obtaining \eqref{weak-max-k2} with constant $-c_1(|\Om|) c_2$.

\end{proof}

 \section*{Appendix A: an analytic proof of Theorem \ref{thm-main1}}

By a version of the uniformization theorem (see Lemma~\ref{unif-annulus}), we have that the annular domain $\Omega$ is isometric to the annulus $\{1< |z|< R\}\subset \mathbb C = \R^2$ endowed with the conformally flat metric $g=e^{2\varphi(z)}dzd\bar z = e^{2\varphi(x,y)}(dx^2 + dy^2)$. Here $R>1$ and $\phi$ is some smooth real function defined on $\{1\leq |z|\leq R\}$.
By the conformal invariance of harmonicity in dimension $2$, in this isometrically equivalent representation $u$ is a solution to 
\[
\begin{cases}
\Delta_0 u=0,\\
u|_{\{|z|=1\}}=t_1,\\
u|_{\{|z|=R\}}=t_2.
\end{cases}
\]
Up to a horizontal translation we can assume that $t_1=0$ and up to a multiplicative constant we can assume that $t_2 = -\ln R$, cf. the corresponding discussion in the proof of Theorem~\ref{thm-main2}). Also, note that both these transformations do not affect the conclusion of the theorem. Accordingly, we can suppose without the loss of generality that $u(z)=-\ln |z|$.

Hence, we may compute $L(t)$ and its derivatives in terms of the conformal factor $\vp$ only. Namely, let $d\si$ denote the $1$-dimensional Hausdorff measure and $d\lambda$ the Lebesgue measure in $\R^2$. Moreover, set $r(z)=|z|$ and $\partial_rf=\langle \nabla_0f, \nabla_0r \rangle$, where $\nabla_0$ and $\langle\cdot,\cdot\rangle$ are the Euclidean gradient and the Euclidean inner product, respectively. Then 
\[
L(t)= \int_{\{|z|=e^{-t}\}}e^{\vp} d\si = e^{-t}\int_{\{|z|=1\}}e^{\varphi(e^{-t}\theta)}\, d\si(\theta).
\]
As observed in Remark \ref{rmk:AG}, at this point Theorem \ref{thm-main1} could be deduced from an abstract result in subharmonic functions theory, \cite[Theorem 3.5.7 ii]{ArGa}. However, since $\varphi$ is smooth, we present here a simple self-contained proof. 
First, we compute
\[\begin{aligned}
L'(t)&=-e^{-t}\int_{\{|z|=1\}}e^{\varphi(e^{-t}\theta)}\, d\si(\theta)-e^{-2t}\int_{\{|z|=1\}}e^{\varphi(e^{-t}\theta)}\partial_r\vp(e^{-t}\theta)\, d\si(\theta)\\
&=-\int_{\{|z|=e^{-t}\}}e^{\varphi(\theta)}\, d\si(\theta)-\int_{\{|z|=e^{-t}\}}re^{\varphi(\theta)}\partial_r\vp(\theta)\, d\si(\theta)\\
&= \int_{\{|z|=e^{-t}\}} r^2e^{2\vp(\theta)}\Big\langle\nabla_0r, \nabla_0\left(\frac{e^{-\vp(\theta)}}{r}\right)\Big \rangle\,d\si(\theta).
\end{aligned}
\]
Moreover,
\[
\begin{aligned}
L''(t)&=\lim_{\epsilon\to 0}\frac{1}{\epsilon}\left[\int_{\{|z|=e^{-t-\epsilon}\}} r^2e^{2\vp}\left \langle \nabla_0r, \nabla_0\left(r^{-1} e^{-\vp}\right)\right \rangle\,d\si - \int_{\{|z|=e^{-t}\}} r^2e^{2\vp}\left \langle\nabla_0r, \nabla_0\left(r^{-1} e^{-\vp}\right)\right \rangle\,d\si\right]\\
&=\lim_{\epsilon\to 0}\frac{1}{\epsilon}\int_{\{e^{-t-\epsilon}\leq|z|\leq e^{-t}\}}-\div_0\left(r^2e^{2\vp}\nabla_0\left(r^{-1} e^{-\vp}\right)\right)\,d\lambda\\
&=\lim_{\epsilon\to 0}\frac{1}{\epsilon}\int_{\{e^{-t-\epsilon}\leq|z|\leq e^{-t}\}} e^{\vp}\left(\frac{1}{r} + 2\partial_r\vp+r|\nabla_0\vp|^2 + r \Delta_0\vp\right)\,d\lambda\\
&=\lim_{\epsilon\to 0}\frac{1}{\epsilon}\int_{\{e^{-t-\epsilon}\leq|z|\leq e^{-t}\}} e^{3\vp}\left( r^3 \left|\nabla_0\left(r^{-1} e^{-\vp}\right)\right|^2-rK\right)\,d\lambda\\
&=\int_{\{|z|=e^{-t}\}}e^{3\vp}\left( r^4 \left|\nabla_0\left(r^{-1} e^{-\vp}\right)\right|^2-r^2K\right)\,d\sigma
\end{aligned}
\]
with $\div_0$ standing for the Euclidean divergence.
Combining the formulas above and using the Cauchy-Schwartz inequality we obtain
\[
\begin{aligned}
(L'(t))^2&\leq \left(\int_{\{|z|=e^{-t}\}} e^{\vp}\,d\si\right)\left( \int_{\{|z|=e^{-t}\}} r^4e^{3\vp}\left| \nabla_0\left(r^{-1} e^{-\vp}\right)\right|^2\,d\si\right)\\
&= L(t)\left(L''(t)+ \int_{\{|z|=e^{-t}\}}e^{3\vp}r^2K\,d\sigma \right),
\end{aligned}
\]
which concludes the proof, since $K\leq 0$.

 \section*{Appendix B: Proof of Theorem~\ref{thm-pdes}}
 For the proof of equation~\eqref{pde1} we closely follow the discussion in~\cite{ww}. Therefore, since the original proof is based on direct computations for $h$ and $K\equiv const$, we will focus on emphasizing differences and novelties in our computations for $k$ (and $h$) and non-constant $K$. Moreover, our approach is slightly more direct and the proof of inequality~\eqref{pde2} is shorter than the corresponding one in the proof of~\cite[Theorem 1.3]{ww}.

Since our aim is to show that equation~\eqref{pde1} holds pointwise, we may localize the discussion and introduce an orthonormal frame in a neighbourhood $U$ of $x\in \Om$. Covariant derivatives are thus computed with respect to this frame. Without the loss of generality we can suppose that
\begin{equation}\label{norm-coord1}
|\nabla u|(x)=u_2(x)>0,\quad u_1(x)=0.
\end{equation}
Moreover, $\det g(x)=1$ and the matrix of metric $g$ at $x$ is the identity. 
The harmonicity of $u$ implies that $u_{22}=-u_{11}$ in $U$. Then, by~\eqref{eq-k-h} one computes that
\[
k=-\frac{u_{11}(u_2^2-u_{1}^2)-2u_1u_2u_{12}}{|\nabla u|^3}:=\frac{f}{|\nabla u|^3}.
\]
 In the chosen coordinate system, the following identities hold at $x$:
\begin{equation}\label{norm-coord2}
 k=-\frac{u_{11}}{u_2},\quad f=-u_2^2u_{11},\quad |\nabla u|_{i}=u_{2i},\quad \left(\frac{1}{|\nabla u|}\right)_{i}=-\frac{u_{2i}}{u_2^2},\quad |\nabla u|_{ii}=\frac{u_{1i}^2}{u_2}+u_{2ii}\quad\hbox{for }i=1,2.
\end{equation}
Let $\phi:=|\nabla u|^{-4}f$. Then 
\[
\Delta \phi=f\Delta(|\nabla u|^{-4})+2\langle \nabla(|\nabla u|^{-4}), \nabla f\rangle +|\nabla u|^{-4}\Delta f:=A+B+C.
\]
Denote the covariant derivatives of $\phi$ with respect to the chosen orthonormal frame as follows:
\[
 \phi_{i}:=(|\nabla u|^{-4})_{i}f+|\nabla u|^{-4}f_{i},\quad \hbox{ for }i=1,2.
\]
Upon direct computations of $\phi_1$ and $\phi_2$ we get the following formulas at $x$
\begin{equation}\label{u-id2}
 u_2u_{111}=-u_2^{3}\phi_1+4u_{11}u_{12},\quad u_2u_{112}=-u_2^{3}\phi_2+2u_{11}u_{22}+2u_{12}^2.
\end{equation}
Notice that by~\eqref{lem-rr1}, harmonicity of $u$ and by symmetries of the Riemann tensor we have the following identities:
\begin{equation}
 u_{111}=-u_{221},\quad u_{121}=u_{112}+u_2R_{2121}=-u_{222}+u_2R_{2121},\quad u_2u_{121}+u_2u_{222}=u_2^2R_{2121}=u_2^2K.  \label{u-id}
\end{equation}
Then, by using computations for $\beta=-2$ in \cite[(3.11)]{ww} we obtain that
\begin{align*}
A&=-u_{11}u_2^2 \Delta (|\nabla u|^{-4})=-16u_2^{-4}(u_{12}^2u_{11}+u_{11}^3)+4u_2^{-4}(u_2u_{121}+u_2u_{222})u_{11}\\
&=-16u_2^{-4}(u_{12}^2u_{11}+u_{11}^3)+4u_2^{-2}u_{11}K.
\end{align*}
Similarly to \cite[(3.14)]{ww}, we obtain that 
\begin{equation*}
B=2\big(f_1 (|\nabla u|^{-4})_1+f_2 (|\nabla u|^{-4})_2\big)=-8u_2^{-1}(u_{21}\phi_1-u_{11}\phi_2)+32u_2^{-4}(u_{11}u_{12}^2+u_{11}^3).
\end{equation*}
Finally, we compute
\begin{align*}
-\Delta f&=-(f_{11}+f_{22})\\
&=u_{1111}u_2^2+2u_2u_{12}u_{111}+u_{2211}u_1^2+2u_1u_{11}u_{221}-2u_{11}^3-4u_1u_{11}u_{111}-2u_{11}u_{12}^2 -6u_1u_{12}u_{121}\\
&-2u_{2}u_{11}u_{121}-2u_1u_2u_{1211}+u_{1122}u_2^2-6u_2u_{11}u_{112}+u_{2222}u_1^2+2u_{11}^3+2u_1u_{12}u_{222}+2u_{12}^2u_{11}\\
&+2u_1u_{11}u_{122}-6u_{2}u_{12}u_{122}-2u_1u_2u_{1222},
\end{align*}
which evaluated at $x$ results in the following formula ($u_1(x)=0$):
\begin{equation}\label{lapl-f}
\Delta f(x)=-u_2^2(u_{1111}+u_{1122})-2u_2u_{12}u_{111}+2u_2u_{11}u_{121}+6u_2u_{11}u_{112}+6u_2u_{12}u_{122}.
\end{equation}
By applying~\eqref{lem-rr2} we have:
\begin{equation}\label{u4-id}
 u_{1111}=-u_{2211}=-u_{1122}+4u_{11}R_{2121}-(\nabla_2\Ric)_{22}u_2.
\end{equation}
Thus, by appealing to~\eqref{lem-rr1}, \eqref{u-id}, \eqref{u-id2} and simplifying the arising expression, we get
\begin{align}
\Delta f(x)=&-u_2^2(4u_{11}K-(\nabla_2\Ric)_{22}u_2)-2u_2u_{12}u_{111}+2u_2u_{11}(u_{112}+u_2K) \nonumber \\
&+6u_2u_{11}u_{112}+6u_2u_{12}(-u_{111}) \nonumber \\
=&-u_2^2(2u_{11}K-(\nabla_2\Ric)_{22}u_2)-8u_2u_{12}u_{111}+8u_2u_{11}u_{112} \label{lapl-f-aux}\\
=&-u_2^2(2u_{11}K-(\nabla_2\Ric)_{22}u_2)+8u_2^{3}(u_{12}\phi_1-u_{11}\phi_2) -16u_{11}u_{12}^2-16u_{11}^3. \nonumber
\end{align}
By adding up $A,B$ and $C=u_2^{-4}\Delta f$ we arrive at assertion~\eqref{pde1}
\begin{equation}\label{phi-k-eq}
\Delta \phi=2u_2^{-2}u_{11}K+u_2^{-1}(\nabla_2\Ric)_{22}=-2\phi K+\frac{(\nabla_2\Ric)_{22}}{|\nabla u|}.
\end{equation}
In order to show equation~\eqref{pde1-star} we use the same orthonormal frame as in the reasoning for $k$. Then 
\begin{equation}\label{h-eq}
 h:=-\div\left(\frac{(u_2,-u_1)}{|\nabla u|}\right)=-\frac{u_{12}(u_1^2-u_2^2)-u_1u_2(u_{11}-u_{22})}{|\nabla u|^3},\quad h(x)=\frac{u_{12}}{u_2}.
 \end{equation}
The proof of~\eqref{pde1-star} follows from direct computations similar to the above reasoning for $|\nabla u|^{-1}k$ and, therefore, we will omit it. The counterpart of~\eqref{phi-k-eq} for $\phi=|\nabla u|^{-1}h$ reads
\begin{align*}
 \Delta \phi&=u_2^{-1}(4u_{12}K+u_2(\nabla_1\Ric)_{22})+u_2^{-3}(-6u_{21}u_2K)=4\frac{u_{12}}{u_2^2}K+\frac{(\nabla_1\Ric)_{22}}{u_2}-6\frac{u_{21}}{u_2^2}\\
 &=-2\phi K-\frac{\langle \nabla K, (u_2,-u_1)\rangle}{|\nabla u|^2}\quad\hbox{at }x.
\end{align*}

Let us now turn to the proof of inequality~\eqref{pde2}. 
We employ the Bochner formula~\eqref{formula-Bochner}, the Kato equality~\eqref{refined-Kato2} and argue as follows
	\begin{align*}
-\Delta\ln k &= -\Delta\ln\phi -\frac 12 \Delta \ln |\nabla u|^2 \\
&= -\frac{\Delta\phi}{\phi} +\frac{|\nabla\phi|^2}{\phi^2}-\frac 12  \frac{\Delta|\nabla u|^2}{|\nabla u|^2} + \frac 12 \frac{|\nabla|\nabla u|^2|^2}{|\nabla u|^4}\\
&=2K - \frac 1\phi \frac{\langle\nabla K,\nabla u\rangle}{|\nabla u|^2}+\frac{|\nabla\phi|^2}{\phi^2} - \frac{|\nabla^2u|^2}{|\nabla u|^2} - K\frac{|\nabla u|^2}{|\nabla u|^2}+ 2 \frac{|\nabla|\nabla u||^2}{|\nabla u|^2}\qquad (\hbox{by }\eqref{pde1} \hbox{ and }~\eqref{formula-Bochner})\\ 
&= K -  \frac{\langle\nabla K,\nabla u\rangle}{k|\nabla u|}+\frac{|\nabla\phi|^2}{\phi^2}\qquad (\hbox{by }\eqref{refined-Kato2}).
 \end{align*}
 This shows estimate~\eqref{pde2} for curvature $k$. Similar reasoning gives the inequality for the curvature of the steepest descent $h$, upon setting $\phi:=\frac{h}{|\nabla u|}$ for $h$ as in~\eqref{h-eq} and repeating the above computations with~\eqref{pde1-star} applied instead of~\eqref{pde1}. 
 

\end{document}